\documentclass[a4paper,11pt]{article}
\usepackage{graphicx}
\usepackage[T1]{fontenc}
\usepackage[utf8]{inputenc}
\usepackage[english]{babel}
\usepackage{amsmath,amsthm,amssymb,amsfonts}
\usepackage{lmodern}
\usepackage{subcaption}
\usepackage{csquotes}
\usepackage[backend=bibtex,style=trad-abbrv,sorting=nyt,url=true]{biblatex}
\renewcommand*{\bibfont}{\small}

\usepackage{geometry}
 \usepackage{empheq}
 \usepackage{xcolor}
 \definecolor{darkgreen}{rgb}{0,0.5,0}
 \definecolor{darkblue}{rgb}{0,0.08,0.45}

\usepackage{hyperref}
\hypersetup{
        colorlinks=true,
        linkcolor=darkblue,
        citecolor=darkblue,
        urlcolor=darkgreen
    }

\usepackage{algorithm}
\usepackage{algorithmic}

\usepackage{enumerate}
\usepackage{cleveref}

\newtheorem{theorem}{Theorem}
\theoremstyle{definition}
\newtheorem{corollary}{Corollary}
\newtheorem{lemma}{Lemma}
\newtheorem{proposition}{Proposition}

\newcommand{\n}[1]{\|#1 \|}
\newcommand{\E}[1]{\mathbb{E}\left[#1\right]}

\renewcommand{\a}{\alpha}
\renewcommand{\b}{\beta}
\renewcommand{\d}{\delta}

\newcommand{\la}{\lambda}

\renewcommand{\phi}{\varphi}
\renewcommand{\t}{\tau}
\newcommand{\s}{\sigma}
\newcommand{\e}{\varepsilon}
\renewcommand{\th}{\theta}
\newcommand{\Th}{\Theta}
\newcommand{\eqdef}{\stackrel{\mathrm{def}}{=}}

\newcommand{\x}{\bar x}
\newcommand{\y}{\bar y}
\newcommand{\z}{\bar z}

\newcommand{\R}{\mathbb R}

\newcommand{\N}{\mathbb N}
\newcommand{\Z}{\mathbb Z}
\newcommand{\cO}{\mathcal O}
\newcommand{\cX}{\mathcal X}

\renewcommand{\empty}{\varnothing}

\newcommand{\lr}[1]{\langle #1\rangle}

\DeclareMathOperator{\clconv}{\overline{conv}}
\def\<#1,#2>{\langle #1,#2\rangle}
\def\toptitlebar{\hrule height1pt \vskip .25in}
\def\bottomtitlebar{\vskip .22in \hrule height1pt \vskip .3in}

\title{\vspace{-1cm}\toptitlebar\textbf{Adaptive Gradient Descent without Descent} \bottomtitlebar}

\author{\textbf{Yura Malitsky}\footnote{EPFL, Lausanne, Switzerland,
        \href{mailto:y.malitsky@gmail.com}{y.malitsky@gmail.com}}
    \and \textbf{Konstantin Mishchenko}\thanks{KAUST, Thuwal, Saudi Arabia,        \href{mailto:konsta.mish@gmail.com}{konsta.mish@gmail.com}}}
\date{}
\begin{document}
\maketitle

\begin{abstract}
    We present a strikingly simple proof that two rules are sufficient
    to automate gradient descent: 1)~don't increase the stepsize too
    fast and 2)~don't overstep the local curvature. No need for
    functional values, no line search, no information about the
    function except for the gradients. By following these rules, you
    get a method adaptive to the local geometry, with convergence
    guarantees depending only on the smoothness in a neighborhood of a
    solution. Given that the problem is convex, our method
    converges even if the global smoothness constant is infinity. As an
    illustration, it can minimize arbitrary continuously
    twice-differentiable convex function. We examine its performance
    on a range of convex and nonconvex problems, including logistic
    regression and matrix factorization.
\end{abstract}

\section{Introduction}
Since the early days of optimization it was evident that there is a
need for algorithms that are as independent from the
user as possible. First-order methods have proven to be versatile and
efficient in a wide range of applications, but one drawback has been
present all that time: the stepsize. Despite certain success
stories, line search procedures and adaptive online methods have not
removed the need to manually tune the optimization parameters. Even in
smooth convex optimization, which is often believed to be much simpler
than the nonconvex counterpart, robust rules for stepsize selection
have been elusive. The purpose of this work is to remedy this deficiency.

The problem formulation that we consider is the basic unconstrained optimization problem
\begin{equation}
  \label{main}
  \min_{x\in \R^d}\ f(x),
\end{equation}
where $f\colon \R^d \to \R$ is  a differentiable
function. Throughout the paper we assume that \eqref{main} has a
solution and we denote its optimal value by $f_*$.

The simplest and most known approach to this problem is the gradient
descent method (GD), whose origin can be traced back to
Cauchy~\cite{cauchy1847methode,lemarechal2012cauchy}. Although it is
probably the oldest optimization method, it continues to play a
central role in modern algorithmic theory and applications. Its
definition can be written in a mere one line,
\begin{equation}
    \label{eq:grad}
    x^{k+1} = x^k - \la \nabla f(x^k), \qquad k\geq 0,
\end{equation}
where $x^0\in \R^d$ is arbitrary and $\la >0$. Under  assumptions that
$ f$ is convex and $L$--smooth (equivalently, $\nabla f$ is $L$-Lipschitz)
that is
\begin{equation}\label{Lipschitz}
    \n{\nabla f(x)-\nabla f(y)}\leq L \n{x-y}, \quad \forall x,y,
\end{equation}
one can show that GD with $\la \in (0, \frac 2 L)$ converges to an
optimal solution~\cite{polyak1963gradient}.  Moreover, with
$\la = \frac 1 L $ the convergence rate~\cite{drori2014performance} is
\begin{equation}
    \label{eq:grad_rate}
    f(x^k)-f_*\leq \frac{L\n{x^0-x^*}^2}{2(2k+1)},
\end{equation}
where $x^*$ is any solution of \eqref{main}. Note that this bound is not improvable~\cite{drori2014performance}.

We identify four important challenges that limit the applications of
gradient descent even in the convex case:
\begin{enumerate}
\itemsep0em
    \item GD is not general: many functions do not satisfy \eqref{Lipschitz} globally.
    \item GD is not a free lunch: one needs to guess $\lambda$,
    potentially trying many values before a success.
    \item GD is not robust: failing to provide $\la < \frac{2}{L}$ may lead to divergence.
    \item GD is slow: even if $L$ is finite, it might be arbitrarily
    larger than local smoothness.
\end{enumerate}

\subsection{Related work}
Certain ways to address some of the issues above already exist in the literature. They include line search, adaptive Polyak's stepsize, mirror descent, dual preconditioning, and stepsize estimation for subgradient methods. We discuss them one by one below, in a process reminiscent of cutting off Hydra's limbs: if one issue is fixed, two others take its place.

The most practical and generic solution to the aforementioned issues
is known as line search (or backtracking). This direction of research
started from the seminal works~\cite{goldstein1962cauchy}
and~\cite{armijo1966} and continues to attract attention, see
\cite{bello2016convergence,salzo2017variable} and references therein.
In general, at each iteration the line search executes another
subroutine with additional evaluations of $\nabla f$ and/or $f$ until
some condition is met. Obviously, this makes each iteration more
expensive.

At the same time, the famous Polyak's
stepsize~\cite{polyak1969minimization} stands out as a very fast
alternative to gradient descent. Furthermore, it does not depend on
the global smoothness constant and uses the current gradient to
estimate the geometry. The formula might look deceitfully simple,
$\la_k = \frac{f(x^k)-f_*}{\n{\nabla f(x^k)}^2}$, but there is a
catch: it is rarely possible to know $f_*$. This method, again,
requires the user to guess $f_*$. What is more, with $\lambda$ it was fine to underestimate it by a factor of 10, but the guess for $f_*$ must be tight, otherwise it has to be reestimated later~\cite{hazan2019revisiting}.

Seemingly no issue is present in the
Barzilai-Borwein stepsize. Motivated by the quasi-Newton
schemes, \cite{barzilai1988two} suggested using steps
\[\la_k = \frac{\lr{x^k-x^{k-1},\nabla f(x^k)-\nabla f(x^{k-1})}}{\n{\nabla f(x^k)-\nabla f(x^{k-1})}^2}.\]
Alas, the convergence results regarding this choice of $\la_k$ are
very limited and the only known case where it provably works is
quadratic problems~\cite{raydan1993barzilai, dai2002r}. In general it
may not work even for smooth strongly convex functions, see the
counterexample in~\cite{burdakov2019stabilized}.

Other more interesting ways to deal with
non-Lipschitzness of $\nabla f$ use the problem structure. The first
method, proposed in~\cite{birnbaum2011distributed} and further
developed in~\cite{Bauschke2016}, shows that the mirror descent method \cite{nemirovsky1983problem}, which is
another extension of GD, can be used with a fixed stepsize, whenever
$f$ satisfies a certain generalization of~\eqref{Lipschitz}.
In addition, \cite{maddison2019dual} proposed the dual
preconditioning method---another refined version of GD. Similarly to
the former technique, it also goes beyond the standard smoothness
assumption of $f$, but in a different way. Unfortunately, these two simple and elegant
approaches cannot resolve all issues yet.
First, not many
functions fulfill respective generalized conditions. And secondly, both
methods still get us back to the problem of not knowing the allowed
range of stepsizes.

A whole branch of optimization considers adaptive extensions of GD that deal with functions whose (sub)gradients are bounded. Probably the earliest work in that direction was written by \cite{shor1962application}. He showed that the method
\begin{align*}
        x^{k+1} = x^k - \la_k \frac{g^k}{\|g^k\|},
\end{align*}
where $g^k\in \partial f(x^k)$ is a subgradient, converges for
properly chosen sequences $(\la_k)$, see, e.g.,\ Section 3.2.3 in~\cite{Nesterov2013}.
Moreover, $\la_k$ requires no knowledge about the function whatsoever.

 Similar methods that work in online setting such as Adagrad~\cite{duchi2011adaptive,mcmahan2010adaptive} received a lot of attention in recent years and remain an active topic of research~\cite{ward19a}.
Methods similar to Adagrad---Adam~\cite{adam, adam2},
RMSprop~\cite{Tieleman2012} and
Adadelta~\cite{zeiler2012adadelta}---remain state-of-the-art for
training neural networks. The corresponding objective is usually
neither smooth nor convex, and the theory often assumes Lipschitzness
of the function rather than of the gradients. Therefore, this
direction of research is mostly orthogonal to ours, although we do
compare with some of these  methods in our neural networks experiment.

We also note that without momentum Adam and RMSprop reduce to signSGD~\cite{bernstein2018signsgd}, which is known to be non-convergent for arbitrary stepsizes on a simple quadratic problem~\cite{karimireddy2019error}.

In a close relation to ours is the recent work~\cite{Malitsky2019},
where there was proposed an adaptive golden ratio algorithm for monotone
variational inequalities. As it solves a more general problem, it does not exploit the structure
of~\eqref{main} and, as most variational inequality methods, has a more conservative update. Although the method estimates the smoothness, it still requires an upper bound on the stepsize as input.

\paragraph{Contribution.} We propose a new version of GD that at no
cost resolves all aforementioned issues. The idea is simple, and it is
surprising that it has not been yet discovered.  In each iteration we
choose $\la_k$ as a certain approximation of the inverse local Lipschitz
constant. With such a choice, we prove that convexity and local
smoothness of $f$ are sufficient for convergence of iterates with the
complexity  $\mathcal{O}(1/k)$ for $f(x^k)-f_*$ in the worst case.

\paragraph{Discussion.} Let us now briefly discuss why we believe that proofs based on monotonicity and global smoothness lead to slower methods.

Gradient descent is by far not a recent method, so there have been
obtained optimal rates of convergence. However, we argue that adaptive
methods require rethinking optimality of the stepsizes. Take as an
example a simple quadratic problem,
 $f(x, y)=\frac{1}{2}x^2+\frac{\delta}{2}y^2$, where
$\delta \ll 1$. Clearly, the smoothness constant of this problem is
equal to $L=1$ and the strong convexity one is $\mu=\delta$. If we run
GD from an arbitrary point $(x^0, y^0)$ with the ``optimal'' stepsize
$\la = \frac{1}{L} =1$, then one iteration of GD  gives us $(x^1, y^1)
= (0, (1-\delta) y^0)$, and similarly $(x^k, y^k) = (0,
(1-\delta)^{k}y^0)$. Evidently for $\delta $ small enough it will take a long time to converge to the solution
$(0,0)$. Instead GD would converge in two iterations if it adjusts its step after the first iteration to $\la = \frac{1}{\delta}$.

Nevertheless, all existing analyses of the gradient descent with $L$-smooth $f$ use
stepsizes bounded by $2/L$. Besides, functional analysis gives
\[f(x^{k+1})\le f(x^k) - \la\Bigl(1 - \frac{\la
    L}{2}\Bigr)\|\nabla f(x^k)\|^2,\]
     from which $1/L$ can be seen as the
``optimal'' stepsize. Alternatively, we can assume that $f$ is
$\mu$-strongly convex, and the analysis in norms gives
\begin{align*}
\|x^{k+1}-x^*\|^2\le \Bigl(1 - 2\frac{\la L \mu}{L + \mu} \Bigr)\|x^k-x^*\|^2 - \la\Bigl(\frac{2}{L+\mu} - \la\Bigr)\|\nabla f(x^k) \|^2,
\end{align*}
whence the ``optimal'' step is $\frac{2}{L+\mu}$.

Finally,
line search procedures use some certain type of monotonicity, for
instance ensuring that $f(x^{k+1})\le f(x^k) - c\|\nabla f(x^k)\|^2$
for some $c>0$. We break with this tradition and merely ask for
convergence in the end.

\section{Main part}
\subsection{Local smoothness of $f$}\label{subs:main}
Recall that a mapping is \emph{locally Lipschitz} if it is Lipschitz over
any compact set of its domain. A function $f$ with (locally) Lipschitz
gradient $\nabla f$ is called (locally) smooth.  It is natural to
ask whether some interesting functions are smooth locally, but not
globally.

It turns out there is no shortage of examples, most prominently among
highly nonlinear functions. In $\R$, they include
$x\mapsto \exp(x)$, $\log(x)$, $\tan(x)$, $x^p$, for $p > 2$,
etc. More generally, they include any twice differentiable $f$, since  $\nabla^2
f(x)$, as a continuous mapping, is  bounded over any bounded set
$\mathcal{C}$. In this case, we have that $\nabla f$ is Lipschitz on
$\mathcal{C}$, due to the mean value inequality
\[\n{\nabla f(x) - \nabla f(y)}\leq \max_{z\in \mathcal{C}}\n{\nabla^2 f(z)}\n{x-y},\quad \forall
    x,y\in \mathcal{C}.\]

Algorithm~\ref{alg:main} that we propose is just a slight modification of
 GD. The quick explanation why local Lipschitzness of $\nabla f$
does not cause us any problems, unlike  most other methods, lies in
the way we prove its convergence. Whenever the stepsize $\la_k$
satisfies two inequalities\footnote{It can be shown that instead of the second condition it is enough to ask for $\la_k^2\le \frac{\n{x^{k}-x^{k-1}}^2}{[3\n{\nabla
        f(x^{k})}^2 - 4\<\nabla f(x^k), \nabla f(x^{k-1})>]_+}$, where $[a]_+\eqdef \max \{0, a\}$, but we prefer the option written in the main text for its simplicity.}
\begin{align*}
        \begin{cases}
  \la_k^2 & \leq (1+\th_{k-1})\la_{k-1}^2,\\ \la_k & \leq \frac{\n{x^{k}-x^{k-1}}}{2\n{\nabla
        f(x^{k})-\nabla f(x^{k-1})}},
        \end{cases}
\end{align*}
independently of the properties of $f$ (apart from convexity), we can show
that the iterates $(x^k)$ remain bounded. Here and everywhere else we
use the convention $1/0=+\infty$, so if $\nabla f(x^k) -\nabla
f(x^{k-1})=0$, the second inequality can be ignored. In the
first iteration it might happen that $\la_1 = \min\{+\infty\}$, in
this case we suppose that any choice of $\la_1>0$ is possible.
 \begin{algorithm}[t]
 \caption{Adaptive gradient descent}
 \label{alg:main}
 \begin{algorithmic}[1]
     \STATE \textbf{Input:} $x^0 \in \R^d$, $\la_0>0$,
     $\th_0=+\infty$\\
     \STATE  $x^1= x^0-\la_0\nabla f(x^0)$
        \FOR{$k = 1,2,\dots$}
        \STATE $\la_k = \min\Bigl\{
        \sqrt{1+\th_{k-1}}\la_{k-1},\frac{\n{x^{k}-x^{k-1}}}{2\n{\nabla
        f(x^{k})-\nabla f(x^{k-1})}}\Bigr\}$
\STATE $x^{k+1} = x^k - \la_k \nabla f(x^k)$
\STATE $\th_k = \frac{\la_k}{\la_{k-1}}$
        \ENDFOR
 \end{algorithmic}
 \end{algorithm}

 Although Algorithm~\ref{alg:main} needs $x^0$ and $\la_0$ as input,
 this is not an issue as one can simply fix $x^0=0$ and
 $\la_0=10^{-10}$. Equipped with a tiny $\la_0$, we ensure
 that $x^1$ will be close enough to $x^0$ and likely will give a good
 estimate for $\la_1$. Otherwise, this has no  influence on further steps.

\subsection{Analysis without descent}
It is now time to show our main contribution, the new analysis technique. The tools that we are going to use are the well-known Cauchy-Schwarz and convexity inequalities. In addition, our methods are related to potential functions~\cite{taylor19a}, which is a powerful tool for producing tight bounds for GD.

Another divergence from the common practice is that our main lemma includes not only $x^{k+1}$ and $x^k$, but also $x^{k-1}$. This can be seen as a two-step analysis, while the majority of optimization methods have one-step bounds. However, as we want to adapt to the local geometry of our objective, it is rather natural to have two terms to capture the change in the gradients.

Now, it is time to derive a characteristic inequality for a specific
Lyapunov energy.
\begin{lemma}\label{lemma:energy}
Let $f\colon \R^d\to \R$ be convex and differential and let $x^*$ be any solution of \eqref{main}.  Then for $(x^k)$ generated
by Algorithm~\ref{alg:main} it holds
\begin{multline}
  \label{eq:lemma_ineq}
\n{x^{k+1}-x^*}^2+ \frac 1 2 \n{x^{k+1}-x^k}^2  + 2\la_{k}(1+\th_{k})
 (f(x^k)-f_*)  \\ \leq \n{x^k-x^*}^2  + \frac 1 2
 \n{x^k-x^{k-1}}^2    + 2\la_k \th_k (f(x^{k-1})-f_*).
\end{multline}
\end{lemma}

\begin{proof}
    Let $k\geq 1$. We start from the standard way of analyzing GD:
    \begin{align*}
      \|x^{k+1}- x^*\|^2
      &= \|x^k - x^*\|^2 + 2\<x^{k+1} - x^{k}, x^k-x^*>+ \|x^{k+1} - x^{k}\|^2\\
      &= \|x^k - x^*\|^2 + 2\la_k \<\nabla f(x^k), x^* - x^k>  + \|x^{k+1} - x^{k}\|^2.
    \end{align*}
    As usually, we bound the scalar product by convexity of $f$:
    \begin{align}\label{eq:conv}
       2\la_k \<\nabla f(x^k), x^* - x^k> \le 2\la_k (f_* - f(x^k)),
    \end{align}
    which gives us
    \begin{align}\label{eq:norms}
          \|x^{k+1}- x^*\|^2
          \le \|x^k - x^*\|^2 - 2\la_k(f(x^k)-f_*)  + \|x^{k+1} - x^{k}\|^2.
    \end{align}
    These two steps have been repeated thousands of times, but now we continue in a completely different manner. We have precisely one ``bad'' term in~\eqref{eq:norms}, which
    is $\n{x^{k+1}-x^k}^2$. We will bound it
    using the difference of gradients:
    \begin{align}\label{dif_x}
      \|x^{k+1} -x^k\|^2 & = 2 \n{x^{k+1}-x^k}^2 -
      \n{x^{k+1}-x^k}^2  = -2\la_k \lr{\nabla f(x^k),
                               x^{k+1}-x^k}- \n{x^{k+1}-x^k}^2\notag
                               \\ &= 2\la_k \lr{\nabla f(x^k)-\nabla
                                   f(x^{k-1}), x^{k}-x^{k+1}}\notag\\
                               & \qquad \qquad +
                                   2\la_k\lr{\nabla f(x^{k-1}), x^k-x^{k+1}}
      -  \n{x^{k+1}-x^k}^2.
  \end{align}
Let us estimate the first two terms in the right-hand
    side above. First,   definition of $\la_k$, followed by
Cauchy-Schwarz and Young's inequalities, yields
\begin{align}\label{cs}
   2\la_k \lr{\nabla f(x^k) -\nabla f(x^{k-1}), x^k - x^{k+1}}  & \leq
2\la_k \n{\nabla f(x^k) -\nabla f(x^{k-1})} \n{x^k - x^{k+1}} \notag \\ &\leq
\n{x^k -x^{k-1}} \n{x^k - x^{k+1}} \notag \\ &\leq
  \frac 1 2 \n{x^{k}-x^{k-1}}^2 + \frac{1}{2}\n{x^{k+1}-x^k}^2.
\end{align}
Secondly, by convexity of $f$,
    \begin{align}    \label{eq:terrible_simple}
        2\la_k\lr{\nabla f(x^{k-1}), x^k-x^{k+1}}
      &=  \frac{2\la_k}{\la_{k-1}}\lr{x^{k-1} - x^{k},
        x^{k}-x^{k+1}}\notag  =  2\la_k\th_k \lr{x^{k-1}-x^{k}, \nabla
        f(x^k)}  \notag \\ & \leq  2\la_k\th_k
    (f(x^{k-1})-f(x^k)).
\end{align}
Plugging~\eqref{cs} and \eqref{eq:terrible_simple}
in~\eqref{dif_x}, we obtain
\begin{align*}
  \n{x^{k+1}-x^k}^2 \leq \frac 1 2 \n{x^k-x^{k-1}}^2 - \frac 1 2
  \n{x^{k+1}-x^k}^2 + 2\la_k\th_k(f(x^{k-1})-f(x^{k})).
\end{align*}
Finally, using the produced estimate for $\n{x^{k+1}-x^k}^2$ in
\eqref{eq:norms}, we deduce the desired inequality~\eqref{eq:lemma_ineq}.
\end{proof}

The above lemma already might give a good hint why our method
works. From inequality~\eqref{eq:lemma_ineq} together with condition
$\la_k^2\leq (1+\th_{k-1})\la_{k-1}^2$, we obtain that the Lyapunov
energy---the left-hand side of \eqref{eq:lemma_ineq}---is
decreasing. This gives us boundedness of $(x^k)$, which is often the
key ingredient for proving convergence. In the next theorem we
formally state our result.
\begin{theorem}\label{th:main}
    Suppose that $f\colon \R^d\to \R$ is convex with locally Lipschitz
    gradient $\nabla f$. Then $(x^k)$ generated by Algorithm~\ref{alg:main} converges to a solution of \eqref{main} and we have
    that
\[f(\hat
    x^k)-f_* \leq \frac{D}{2S_k}=\mathcal{O}\Bigl(\frac{1}{k}\Bigr),\]
where
\begin{align*}
\hat x^k &=  \frac{\la_k(1+\th_k)x^k +
        \sum_{i=1}^{k-1}w_i
           x^i}{S_k},\\
    w_i &= \la_i(1+\th_i)-\la_{i+1}\th_{i+1},\\
  S_k &= \la_k(1+\th_k) + \sum_{i=1}^{k-1}w_i = \sum_{i=1}^k \la_i + \la_1\th_1,
\end{align*}
and $D$ is a constant that explicitly depends on the initial data and
the solution set, see \eqref{eq:telescope}.

\end{theorem}
Our proof will consist of two parts. The first one is a
straightforward application of Lemma~\ref{lemma:energy}, from which we
derive boundedness of $(x^k)$ and complexity result. Due to its
conciseness, we provide it directly after this remark. In the second
part, we prove that the whole sequence $(x^k)$ converges to a
solution. Surprisingly, this part is a bit more technical than
expected, and thus we postpone it to the appendix.

\begin{proof} \textit{(Boundedness and complexity result.)}

    Fix any $x^*$ from the solution set of \cref{main}.  Telescoping
    inequality~\eqref{eq:lemma_ineq}, we deduce
    \begin{multline}\label{eq:telescope}
        \n{x^{k+1}-x^*}^2+ \frac 1 2 \n{x^{k+1}-x^k}^2  + 2\la_k
  (1+\th_k) (f(x^k)-f_*) \\+
  2\sum_{i=1}^{k-1}[\la_i(1+\th_i)-\la_{i+1}\th_{i+1}](f(x^i)-f_*) \\
  \leq \, \n{x^1-x^*}^2  + \frac 1 2 \n{x^1-x^{0}}^2   + 2\la_1 \th_1 [f(x^{0})-f_*]\eqdef D.
\end{multline}
Note that by definition of $\la_k$, the second line above is always
nonnegative. Thus, the sequence $(x^k)$ is bounded. Since $\nabla f$
is locally Lipschitz, it is Lipschitz continuous on bounded sets. It
means that for the set
$\mathcal{C} = \clconv \{x^*, x^0, x^1,\dots\}$, which is bounded as the convex hull of bounded points, there exists $L>0$ such that
\[\n{\nabla f(x)-\nabla f(y)} \leq L \n{x-y} \quad \forall
    x,y\in \mathcal{C}.\]
    Clearly, $\la_1=\frac{\n{x^1-x^0}}{2\n{\nabla
        f(x^1)-\nabla f(x^0)}}\geq
\frac{1}{2L}$, thus, by induction one can prove that
\(\la_k \geq \frac{1}{2L} \), in other words, the sequence
$(\la_k)$ is separated from zero.

Now we want to apply the Jensen's inequality for the sum of all terms
$f(x^i)-f_*$ in the left-hand side of \eqref{eq:telescope}.  Notice, that the
total sum of coefficients at these terms is
\[\la_k(1+\th_k)
    +\sum_{i=1}^{k-1}[\la_i(1+\th_i)-\la_{i+1}\th_{i+1}]
    =\sum_{i=1}^k \la_i + \la_1\th_1=S_k\]
Thus, by Jensen's inequality,
\[\frac D 2 \geq \frac{\text{LHS of \eqref{eq:telescope}}}{2} \geq S_k (f(\hat
    x^k)-f_*),\]
where $\hat x^k$ is given in the statement of the theorem.
By this, the first part of the proof is complete. Convergence of $(x^k)$ to a solution is provided in the appendix.
\end{proof}
As we have shown that $\la_i\geq \frac{1}{2L}$ for all $i$, we have a
theoretical upper bound $f(\hat x^k)-f_*\leq \frac{D L}{k}$. Note that
in practice, however, $(\la_k)$ might be much larger than the
pessimistic lower bound $\frac{1}{2L}$, which we observe in our experiments
together with a faster convergence.

\subsection{$f$ is locally strongly convex}
Since one of our goals is to make optimization easy to use, we believe that
a good method should have state-of-the-art guarantees in various scenarios. For strongly convex functions, this means that we want to see linear convergence, which is not covered by normalized GD or online methods. In
\cref{subs:main} we have shown that Algorithm~\ref{alg:main} matches
the $\mathcal{O}(1/\e)$ complexity of GD on convex problems. Now we show that it also matches
$\mathcal{O}(\frac{L}{\mu}\log\frac{1}{\e})$ complexity of GD when $f$
is locally strongly convex. Similarly to local smoothness, we
call $f$ \emph{locally strongly convex} if it is strongly convex over any
compact set of its domain.

For proof simplicity, instead of using bound
$\la_k \leq \sqrt{1+\th_{k-1}}\la_{k-1}$ as in step~4 of
Algorithm~\ref{alg:main} we will use a more conservative bound
$\la_k\leq \sqrt{1+\frac{\th_{k-1}}{2}}\la_{k-1}$ (otherwise the
derivation would be too technical). It is clear that with such a
change \Cref{th:main} still holds true, so the sequence is bounded and we can rely on local smoothness and local strong convexity.
\begin{theorem}\label{th:strong}
    Suppose that $f\colon \R^d\to \R$ is locally strongly convex and
    $\nabla f$ is locally Lipschitz.  Then $(x^k)$ generated by
    Algorithm~\ref{alg:main} (with the modification mentioned above)
    converges to the solution $x^*$ of \eqref{main}. The complexity to
    get $\|x^k - x^*\|^2\le \varepsilon$ is
    $\mathcal{O}(\kappa \log\frac 1 \e)$, where
    $\kappa = \frac{L}{\mu}$ and $L,\mu$ are the smoothness and strong
    convexity constants of $f$ on the set
    $\mathcal{C}=\clconv\{x^*, x^0, x^1, \dots\}$.
\end{theorem}

We want to highlight that in our rate $\kappa$ depends on the local
Lipschitz and strong convexity constants $L$ and $\mu$, which is
meaningful even when these properties are not satisfied
globally. Similarly, if $f$ is globally smooth and strongly convex,
our rate is still faster as it depends on the smaller local constants.

\section{Heuristics}
In this section, we describe several extensions of our method. We do
not have a full theory for them, but believe that they are of interest
in applications.
\subsection{Acceleration}
\begin{algorithm}[t]
 \caption{Adaptive accelerated gradient descent}
 \label{alg:accel}
 \begin{algorithmic}[1]
     \STATE \textbf{Input:} $x^0 \in \R^d$, $\la_0>0$, $\Lambda_0>0$,
     $\th_0=\Th_0 = +\infty$\\
     \STATE $y^1 = x^1= x^0-\la_0\nabla f(x^0)$
        \FOR{$k = 1,2,\dots$}
        \STATE $\la_k = \min\Bigl\{
        \sqrt{1+\frac{\th_{k-1}}{2}}\la_{k-1},\frac{\n{x^{k}-x^{k-1}}}{2\n{\nabla
                f(x^{k})-\nabla f(x^{k-1})}}\Bigr\}$
        \STATE $\Lambda_k = \min\Bigl\{\sqrt{1 +
                \frac{\Th_{k-1}}{2}}\Lambda_{k-1}, \frac{\| \nabla f(x^k) -
                \nabla f(x^{k-1}) \|}{2\|x^k - x^{k-1}\|} \Bigr\}$
        \STATE $\beta_k =\frac{\sqrt{1/\lambda_k} - \sqrt{\Lambda_k}}{\sqrt{1/\lambda_k}+\sqrt{\Lambda_k}} $
        \STATE $y^{k+1} = x^k - \la_k \nabla f(x^k)$
        \STATE $x^{k+1}=y^{k+1} + \beta_k (y^{k+1}-y^k)$
\STATE $\th_k = \frac{\la_k}{\la_{k-1}}$, $\Th_k = \frac{\Lambda_k}{\Lambda_{k-1}}$
        \ENDFOR
 \end{algorithmic}
 \end{algorithm}

Suppose that $f$ is $\mu$-strongly convex.
One version of the  accelerated gradient method proposed by
Nesterov~\cite{Nesterov2013} is
\begin{align*}
        y^{k+1} &= x^k  - \frac{1}{L}\nabla f(x^k),\\
        x^{k+1}&= y^{k+1} + \beta(y^{k+1} - y^{k}),
\end{align*}
where $\beta = \frac{\sqrt{L} - \sqrt{\mu}}{\sqrt{L}+\sqrt{\mu}}$.
Adaptive gradient descent for
 strongly convex $f$ efficiently estimated $\frac{1}{2L}$ by
\begin{align*}
        \lambda_k = \min\biggl\{\sqrt{1 + \frac{\th_{k-1}}{2}}\lambda_{k-1}, \frac{\|x^k - x^{k-1}\|}{2\|\nabla f(x^k) -\nabla f(x^{k-1})\|}\biggr\}.
\end{align*}
What about the strong convexity constant $\mu$?  We know that it
 equals to the inverse smoothness constant of
the conjugate $f^*(y) \eqdef \sup_x\{\lr{x,y} - f(x) \}$. Thus, it is
tempting to estimate this inverse constant just as we estimated inverse
smoothness of $f$, i.e.,\ by formula
\begin{align*}
        \Lambda_k = \min\biggl\{\sqrt{1 +\frac{\Th_{k-1}}{2}} \Lambda_{k-1}, \frac{\| p^k - p^{k-1} \|}{2\|\nabla f^*(p^k) - \nabla f^*(p^{k-1})\|} \biggr\}
\end{align*}
where $p^k$ and $p^{k-1}$ are some elements of the dual space and
$\Th_k = \frac{\Lambda_k}{\Lambda_{k-1}}$. A natural choice then is $p^k = \nabla f(x^k)$ since it is an element of the dual space that we use. What is its value? It is well known that $\nabla f^*(\nabla f(x))=x$, so we come up with the update rule
\begin{align*}
        \Lambda_k = \min\biggl\{\sqrt{1 + \frac{\Th_{k-1}}{2}}\Lambda_{k-1}, \frac{\| \nabla f(x^k) - \nabla f(x^{k-1}) \|}{2\|x^k - x^{k-1}\|} \biggr\},
\end{align*}
and hence we can estimate $\beta$ by $
        \beta_k = \frac{\sqrt{1/\lambda_k} - \sqrt{\Lambda_k}}{\sqrt{1/\lambda_k}+\sqrt{\Lambda_k}}$.

We summarize our arguments in
Algorithm~\ref{alg:accel}. Unfortunately, we do not have any
theoretical guarantees for it.

 Estimating strong convexity parameter $\mu$ is important in
 practice. Most common approaches rely on restarting technique
 proposed by \cite{Nesterov2013a}, see
 also~\cite{fercoq2017adaptive} and references therein. Unlike
 Algorithm~\ref{alg:accel}, these works have theoretical guarantees,
 however, the methods themselves are more complicated and still
 require tuning of other unknown parameters.

 \subsection{Uniting our steps with stochastic gradients}
Here we would like to discuss applications of our method to the problem
\begin{align*}
        \min_x \E{f_\xi(x)},
\end{align*}
where $f_\xi$ is almost surely $L$-smooth and $\mu$-strongly convex. Assume that at each iteration we get sample $\xi^k$ to make a stochastic gradient step,
\begin{align*}
x^{k+1}=x^k - \la_k\nabla f_{\xi^k}(x^k).
\end{align*}
Then, we have two ways of incorporating our stepsize into SGD.
The first is to reuse $\nabla f_{\xi^k}(x^k)$ to estimate $L_k=\frac{\n{\nabla f_{\xi^k}(x^{k})-\nabla
        f_{\xi^k}(x^{k-1})}}{\n{x^{k}-x^{k-1}}}$, but this would make
        $\la_k\nabla f_{\xi^k}(x^k)$ biased. Alternatively, one can use an extra
        sample to estimate $L_k$, but this is less intuitive since our goal is to estimate
        the curvature of the function used in the update.

        We give a full description in~\Cref{alg:stoch}. We remark that
        the option with a biased estimate performed much better in our
        experiments with neural networks. The theorem below provides
        convergence guarantees for both cases, but with different
        assumptions.

\begin{algorithm}[t]
 \caption{Adaptive SGD}
 \label{alg:stoch}
 \begin{algorithmic}[1]
     \STATE \textbf{Input:} $x^0 \in \R^d$, $\la_0>0$,
     $\th_0=+\infty$, $\xi^0$, $\alpha>0$
     \STATE $x^1 = x^0-\la_0\nabla f_{\xi^0}(x^0)$
        \FOR{$k = 1,2,\dots$}
        \STATE Sample $\xi^k$ and optionally $\zeta^k$
        \STATE Option I (biased):  $L_k = \frac{\n{\nabla
        f_{\xi^k}(x^{k})-\nabla f_{\xi^k}(x^{k-1})}}{\n{x^{k}-x^{k-1}}}$
        \STATE Option II (unbiased):  $L_k = \frac{\n{\nabla
        f_{\zeta^k}(x^{k})-\nabla f_{\zeta^k}(x^{k-1})}}{\n{x^{k}-x^{k-1}}}$
        \STATE $\la_k = \min\Bigl\{
        \sqrt{1+\th_{k-1}}\la_{k-1},\frac{\alpha}{ L_k}\Bigr\}$
\STATE $x^{k+1} = x^k - \la_k \nabla f_{\xi^k}(x^k)$
\STATE $\th_k = \frac{\la_k}{\la_{k-1}}$
        \ENDFOR
 \end{algorithmic}
 \end{algorithm}

\begin{theorem}
        Let $f_\xi$ be $L$-smooth and $\mu$-strongly convex almost surely. Assuming $\alpha\le \frac{1}{2\kappa}$ and estimating $L_k$ with $\nabla f_{\zeta^k}$, the complexity to get $\E{\n{x^k- x^*}^2}\le \varepsilon$ is not worse than $\mathcal{O}\left(\frac{\kappa^2}{\varepsilon}\log \frac{\kappa}{\varepsilon}\right)$. Furthermore, if the model is overparameterized, i.e.,\ $\nabla f_\xi(x^*)=0$ almost surely, then one can estimate $L_k$ with $\xi^k$ and the complexity is $\mathcal{O}\left(\kappa^2 \log \frac{1}{\varepsilon}\right)$.
\end{theorem}
Note that in both cases we match the known dependency on $\varepsilon$ up to logarithmic terms, but we get an extra $\kappa$ as the price for adaptive estimation of the stepsize.

Another potential application of our techniques is estimation of decreasing stepsizes in SGD. The best known rates for SGD~\cite{stich2019unified}, are obtained using $\la_k$ that evolves as $\mathcal{O}\left(\frac{1}{L+\mu k}\right)$. This requires estimates of both smoothness and strong convexity, which can be borrowed from the previous discussion. We leave rigorous proof of such schemes for future work.
\section{Experiments}
In the experiments\footnote{See
  \href{https://github.com/ymalitsky/adaptive_GD}{https://github.com/ymalitsky/adaptive\_gd}},
we compare our approach with the two most related methods: GD and
Nesterov's accelerated method for convex
functions~\cite{Nesterov1983a}. Additionally, we consider line search,
Polyak step, and Barzilai-Borwein method. For neural networks we also
include a comparison with SGD, SGDm and Adam.

\paragraph{Logistic regression.}
The  logistic loss with $\ell_2$-regularization is given by
$\frac{1}{n}\sum_{i=1}^n \log(1 + \exp(-b_i a_i^\top x)) +
\frac{\gamma}{2}\|x\|^2$, where $n$ is the number of observations,
$\gamma>0$ is a regularization parameter, and
$(a_i, b_i)\in\R^{d}\times \R$, $i=1,\dots, n$, are the observations. We
use `mushrooms' and `covtype' datasets to run the experiments. We choose $\gamma$
proportionally to $\frac{1}{n}$ as often done in practice. Since we
have closed-form expressions to estimate $L=\frac{1}{4n}\|A\|^2+\gamma$, where $A=(a_1^\top, \dotsc, a_n^\top)^\top$, we used stepsize $\frac{1}{L}$ in GD and its acceleration. The
results are provided in Figure~\ref{fig:logistic}.

\paragraph{Matrix factorization.}

Given a matrix $A\in \R^{m\times n}$ and $r<\min\{m,n\}$, we want to
solve $\min_{X=[U,V]} f(X)=f(U,V)=\frac 1 2 \n{UV^\top-A}^2_F$ for
$U\in \R^{m\times r}$ and $V\in \R^{n\times r}$.  It is a nonconvex
problem, and the gradient $\nabla f$ is not globally
Lipschitz. With some tuning, one still can apply GD and Nesterov's
accelerated method, but---and we want to emphasize it---it was not a
trivial thing to find the steps in practice.  The steps we have chosen
were almost optimal, namely, the methods did not  converge if we
doubled the steps.  In contrast, our methods do
not require any tuning, so even in this regard they are much more
practical. For the experiments we used Movilens 100K
dataset~\cite{harper2016movielens} with more than million entries and
several values of $r=10,\ 20,\ 30$. All algorithms were initialized at
the same point, chosen randomly. The results are presented in
\Cref{fig:factorization}.

\paragraph{Cubic regularization.}
\begin{figure*}[!t]
\centering
\begin{subfigure}[t]{0.32\textwidth}
    {\includegraphics[trim={3mm 0 3mm 0},clip,width=1\textwidth]{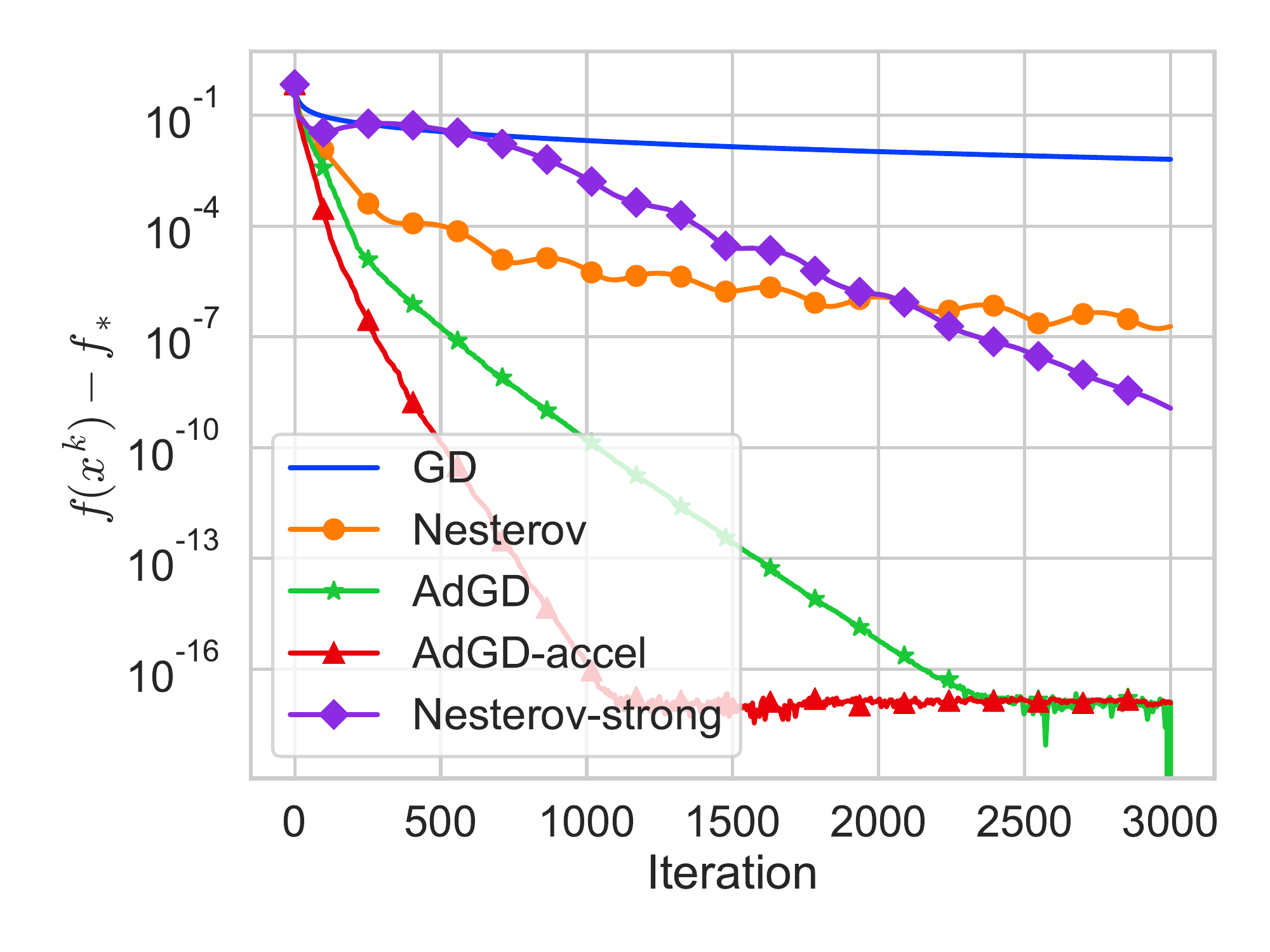}}
\caption{Mushrooms dataset, objective}
\end{subfigure}
\hfill
\begin{subfigure}[t]{0.32\textwidth}
{\includegraphics[trim={3mm 0 3mm 0},clip,width=1\textwidth]{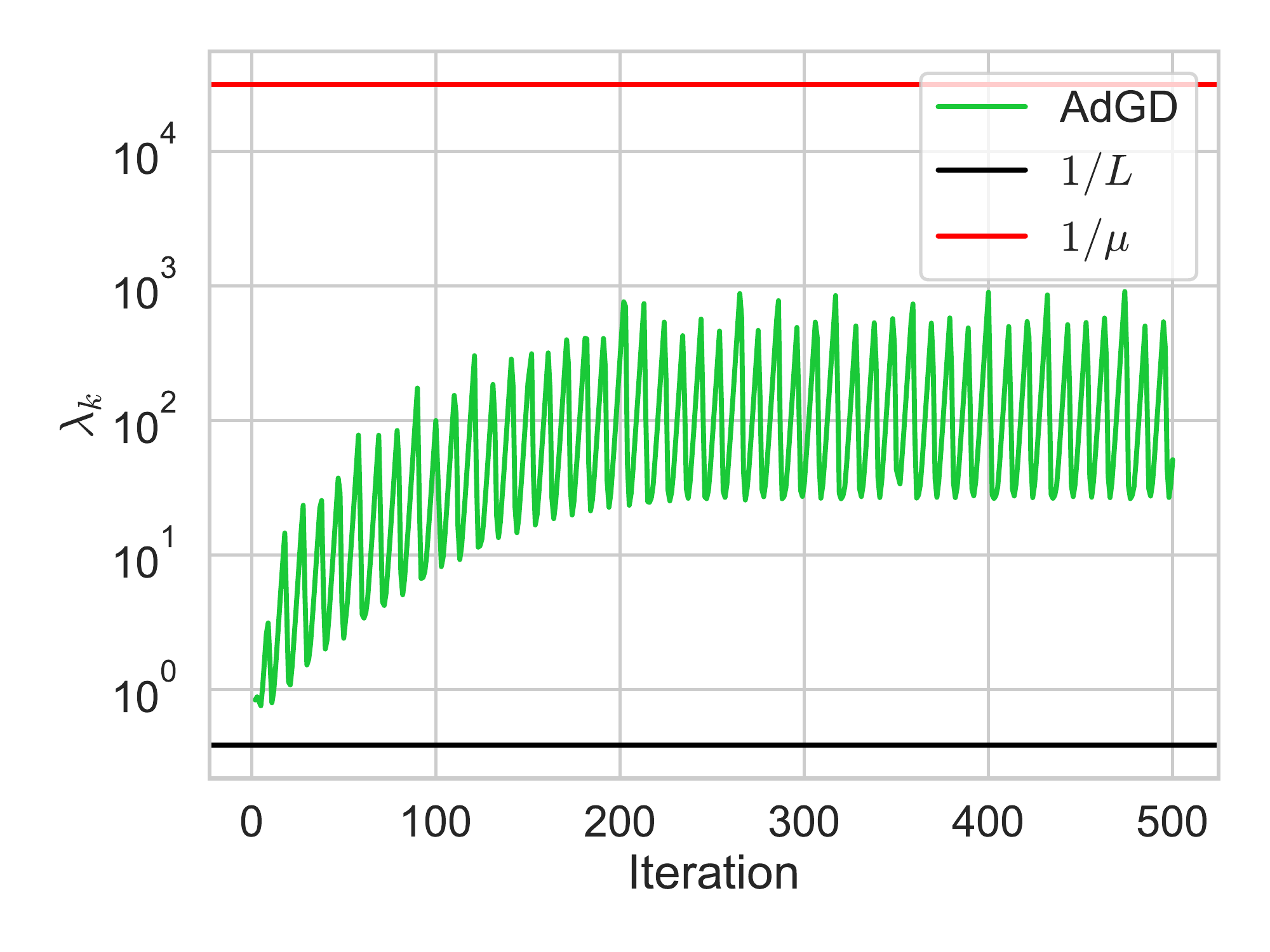}}
\caption{Mushrooms dataset, stepsize}
\end{subfigure}
\hfill
\begin{subfigure}[t]{0.32\textwidth}
{\includegraphics[trim={3mm 0 3mm 0},clip,width=1\textwidth]{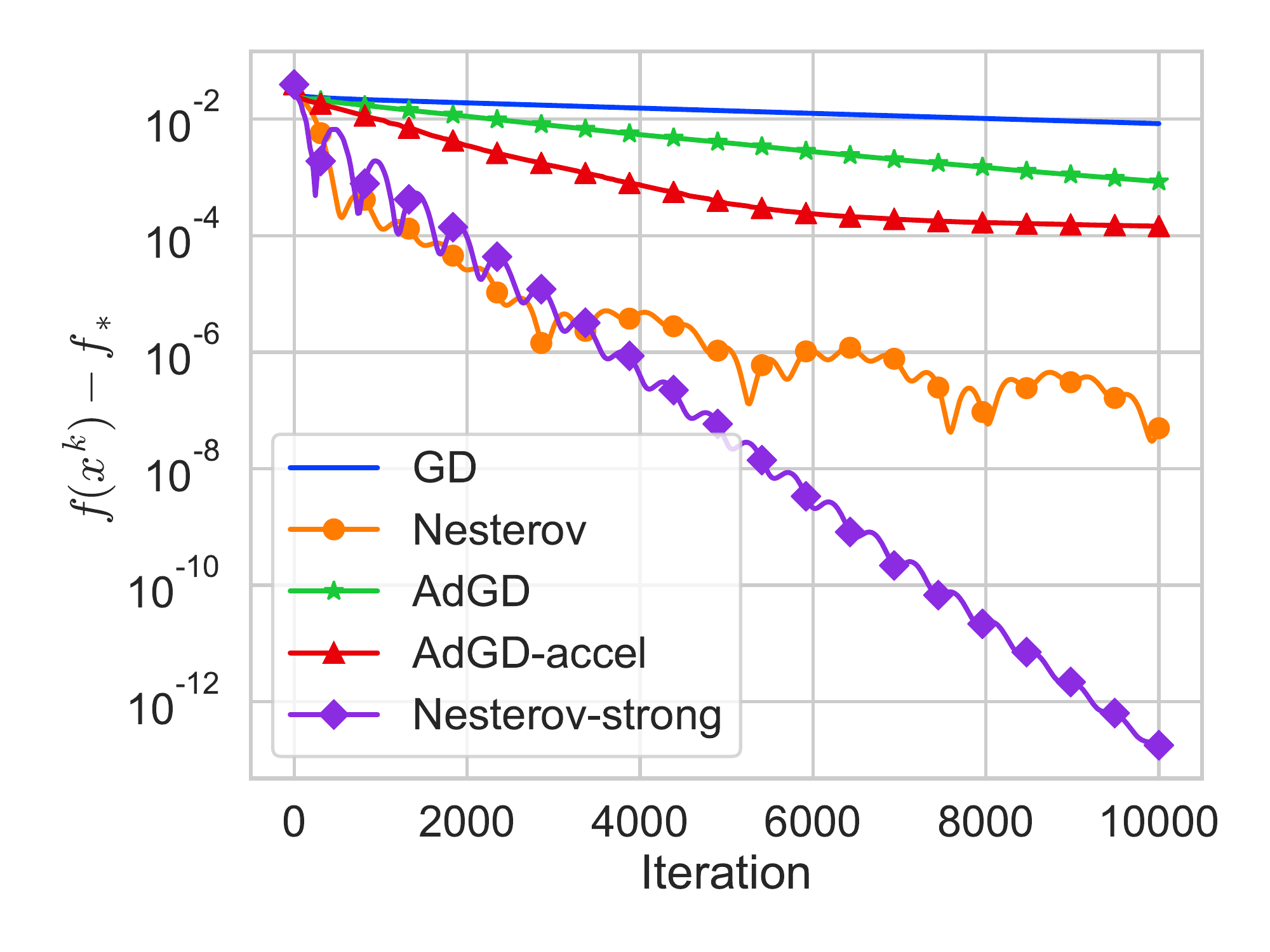}}
\caption{Covtype dataset, objective}
\end{subfigure}
\caption{Results for the logistic regression problem.}
\label{fig:logistic}
\end{figure*}

\begin{figure*}[t]
\centering
\begin{subfigure}[t]{0.32\textwidth}
    {\includegraphics[trim={3mm 0 3mm 0},clip,width=1\textwidth]{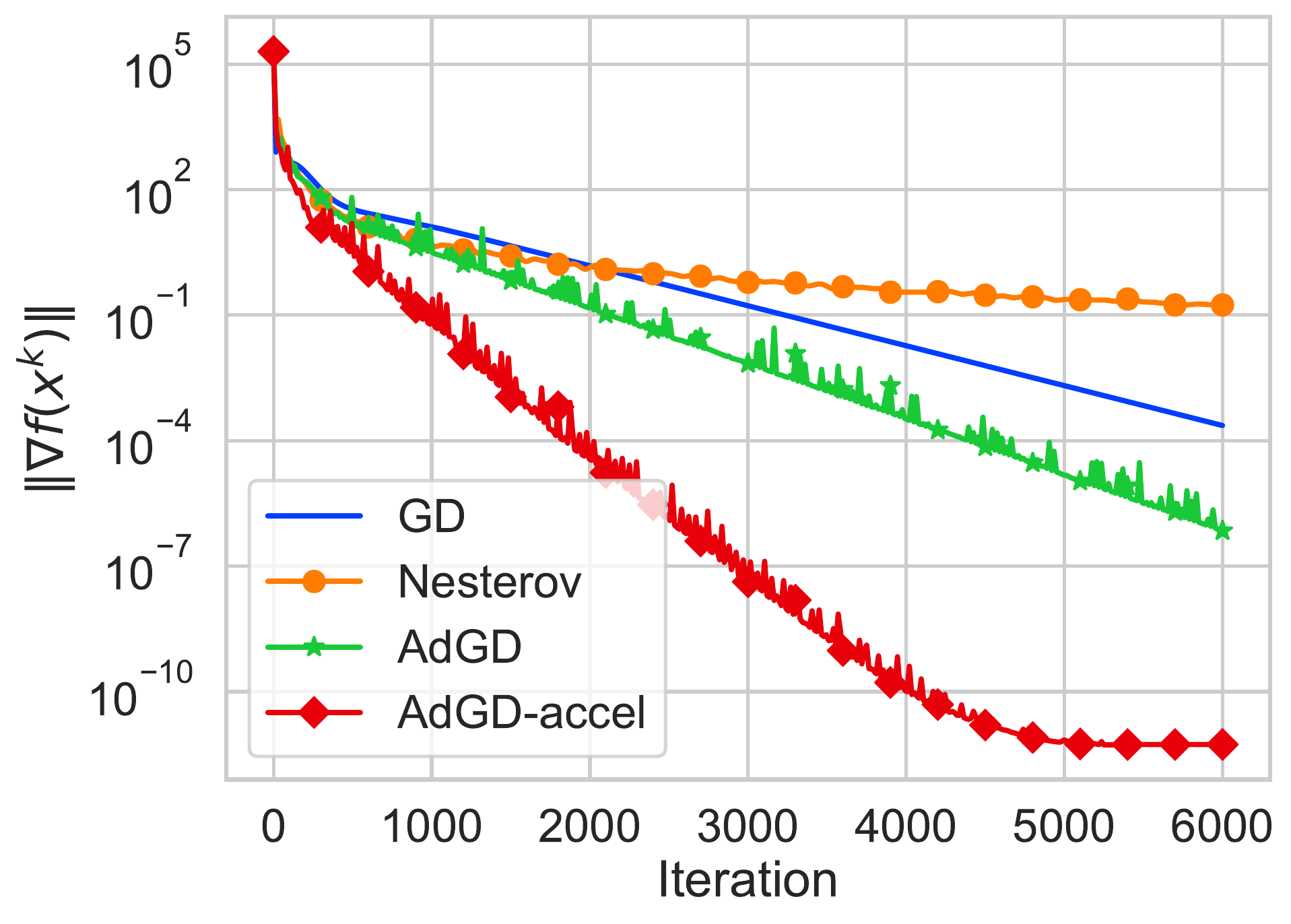}}
\caption{$r=10$}
\end{subfigure}
\hfill
\begin{subfigure}[t]{0.32\textwidth}
{\includegraphics[trim={3mm 0 3mm 0},clip,width=1\textwidth]{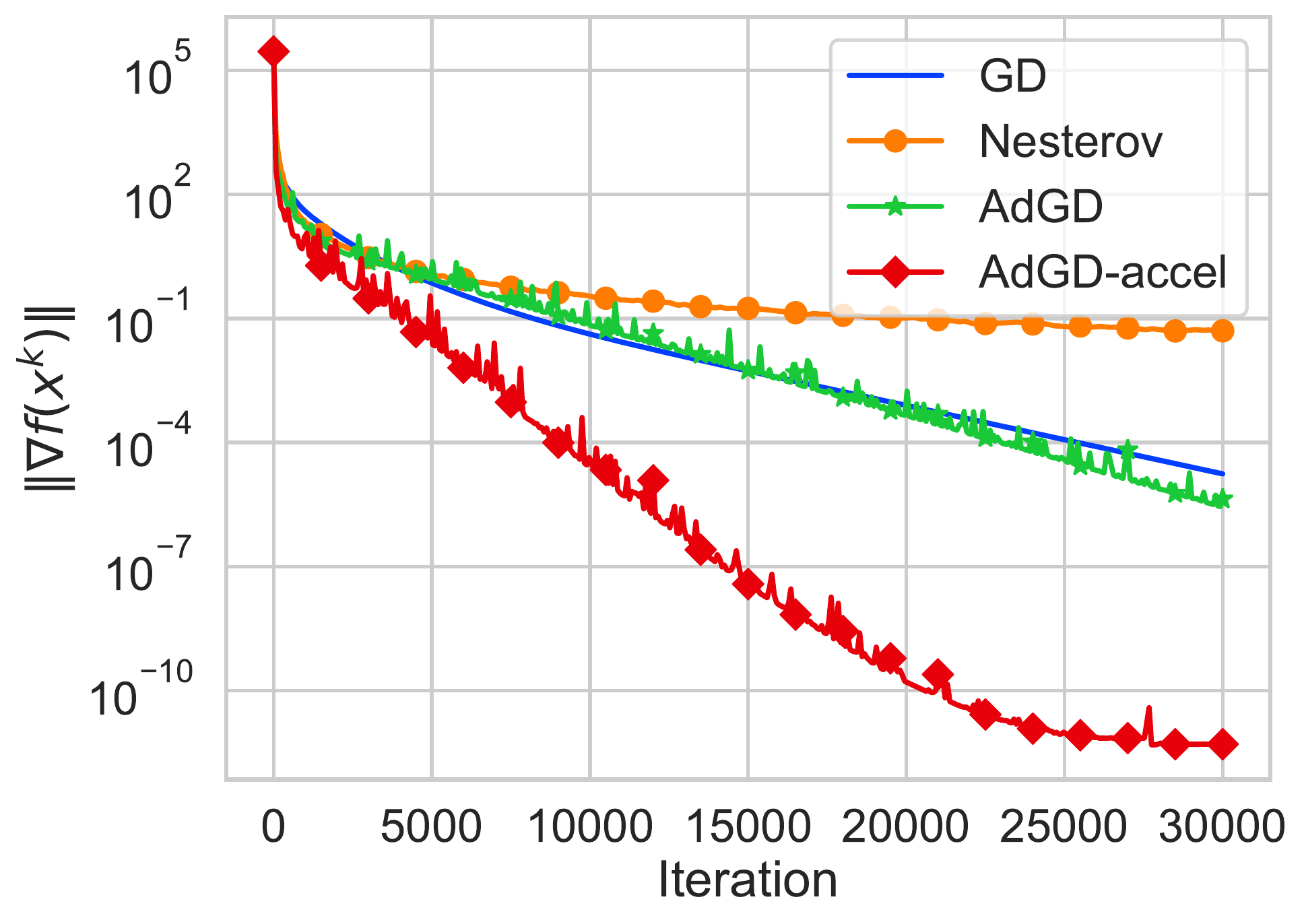}}
\caption{$r=20$}
\end{subfigure}
\hfill
\begin{subfigure}[t]{0.32\textwidth}
{\includegraphics[trim={3mm 0 3mm 0},clip,width=1\textwidth]{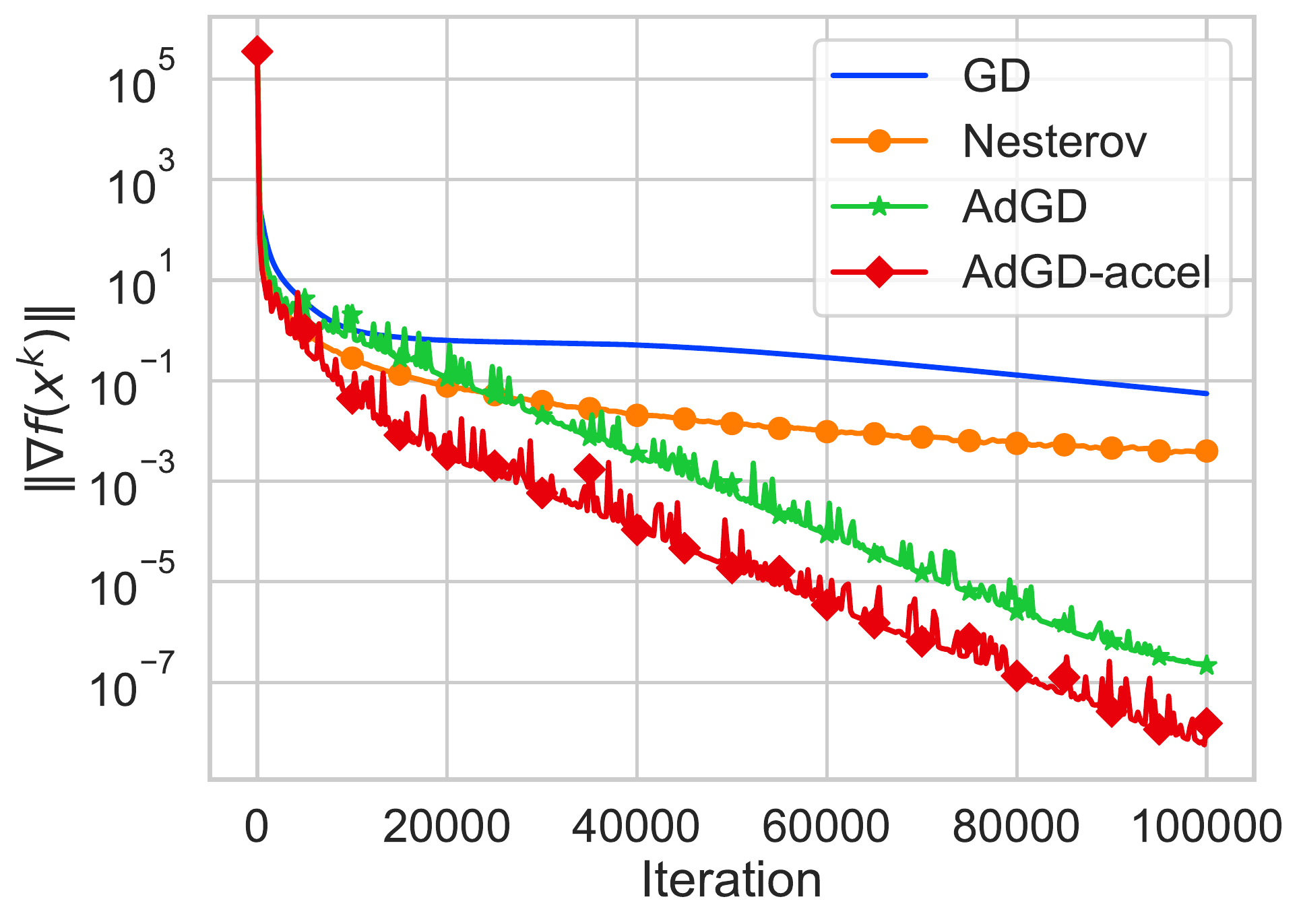}}
\caption{$r=30$}
\end{subfigure}
\caption{Results for matrix factorization. The objective is neither convex nor
    smooth.}
\label{fig:factorization}
\end{figure*}

\begin{figure*}[t]
\centering
\begin{subfigure}[t]{0.32\textwidth}
    {\includegraphics[trim={3mm 0 3mm 0},clip,width=1\textwidth]{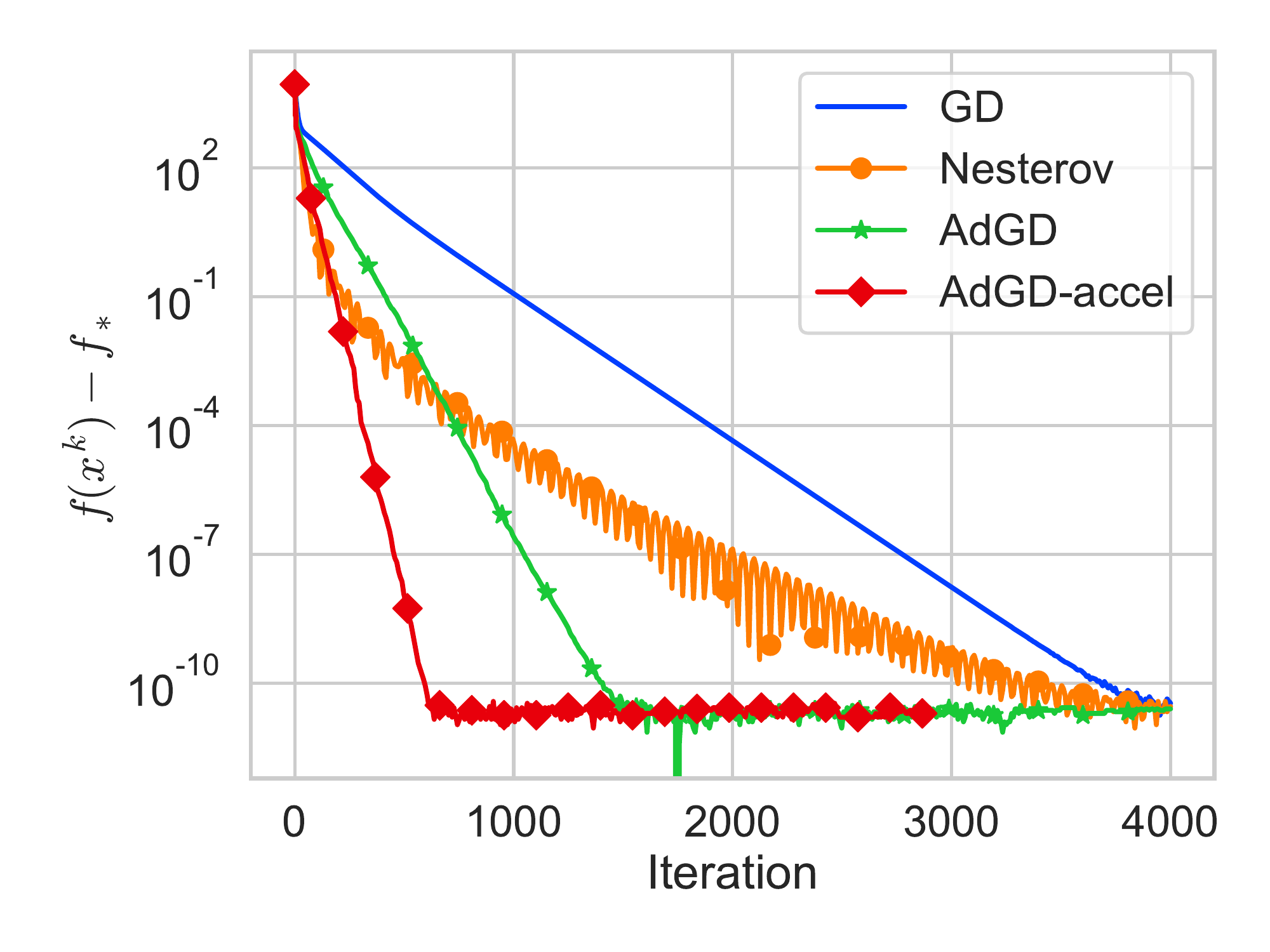}}
\caption{$M=10$}
\end{subfigure}
\hfill
\begin{subfigure}[t]{0.32\textwidth}
{\includegraphics[trim={3mm 0 3mm 0},clip,width=1\textwidth]{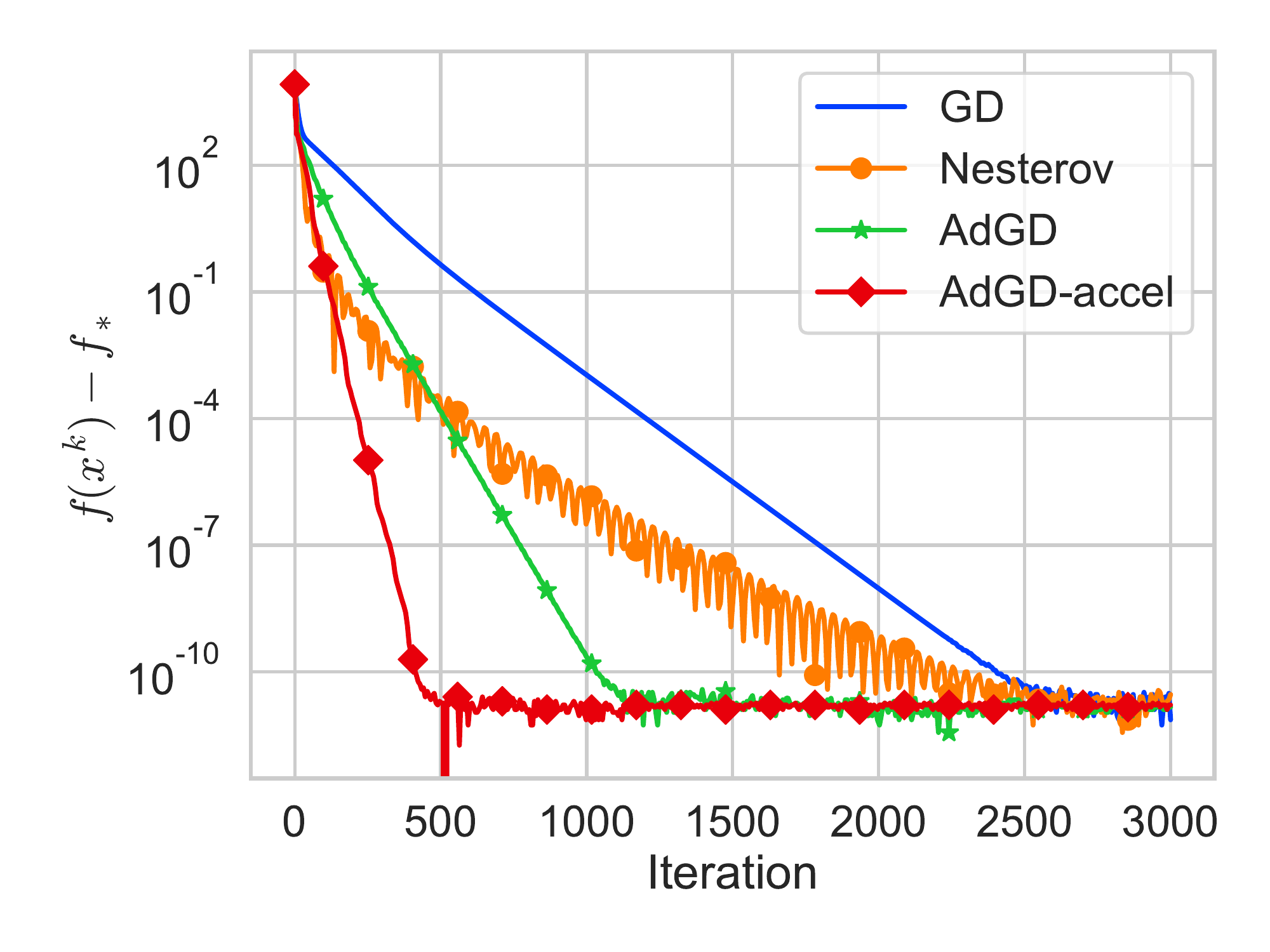}}
\caption{$M=20$}
\end{subfigure}
\hfill
\begin{subfigure}[t]{0.32\textwidth}
{\includegraphics[trim={3mm 0 3mm 0},clip,width=1\textwidth]{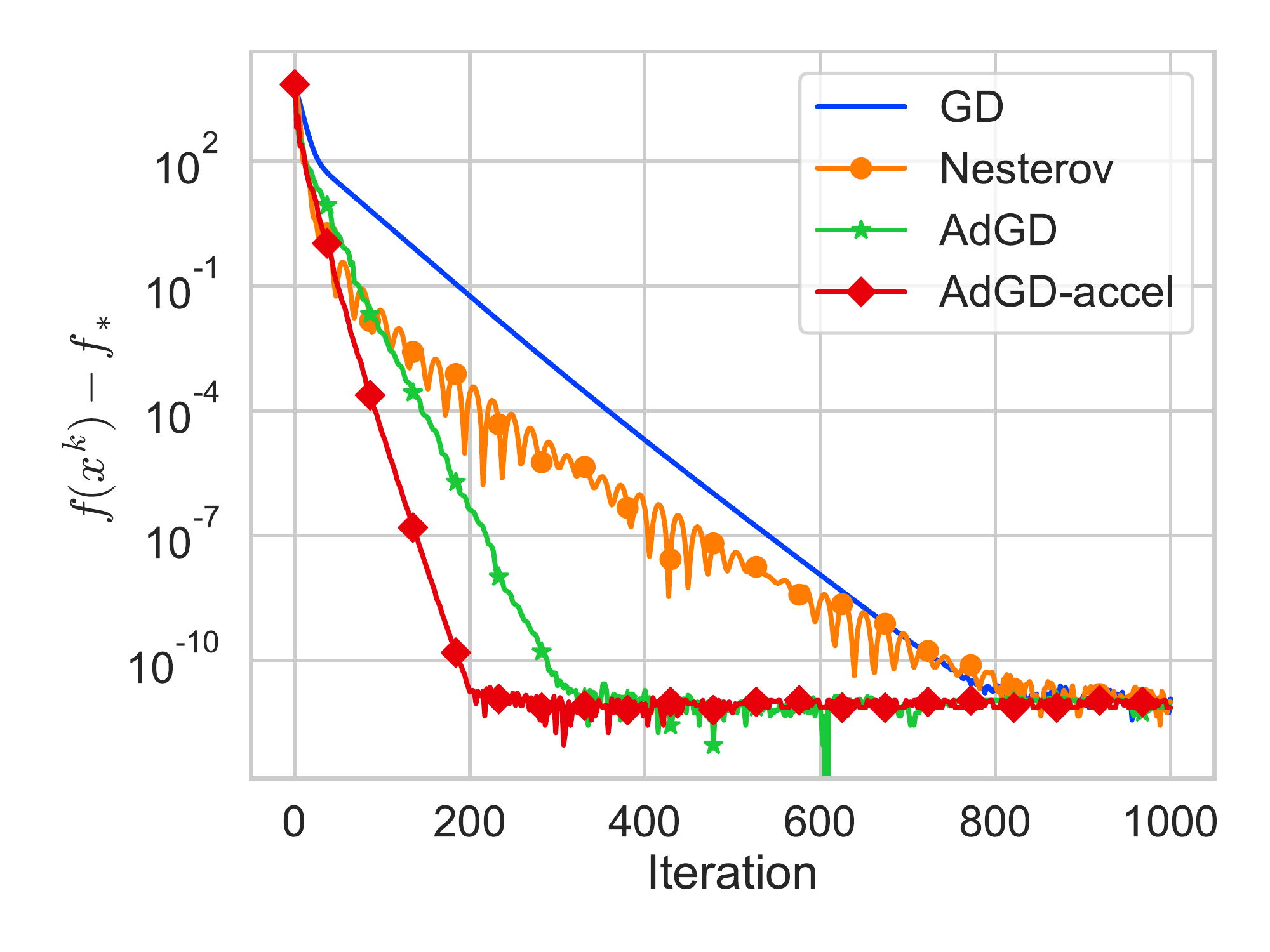}}
\caption{$M=100$}
\end{subfigure}
\caption{Results for the non-smooth subproblem from cubic regularization.}
\label{fig:cubic}
\end{figure*}

In cubic regularization of Newton method~\cite{nesterov2006cubic}, at
each iteration we need to  minimize $f(x)= g^\top x +
\frac{1}{2}x^\top H x + \frac{M}{6}\|x\|^3$, where $g\in \R^d, H\in
\R^{d\times d}$ and $M>0$ are given. This objective is smooth only
locally due to the cubic term, which is our motivation to consider
it. $g$ and $H$ were the gradient and the Hessian of the logistic loss
with the `covtype' dataset, evaluated at $x=0\in\R^d$. Although the values of
$M=10$, $20$, $100$ led to similar results, they also required different numbers of iterations, so we present
the corresponding results in \Cref{fig:cubic}.

\paragraph{Barzilai-Borwein, Polyak and line searches.}
\begin{figure*}[t]
\centering
\begin{subfigure}[t]{0.32\textwidth}
    {\includegraphics[trim={3mm 0 3mm 0},clip,width=1\textwidth]{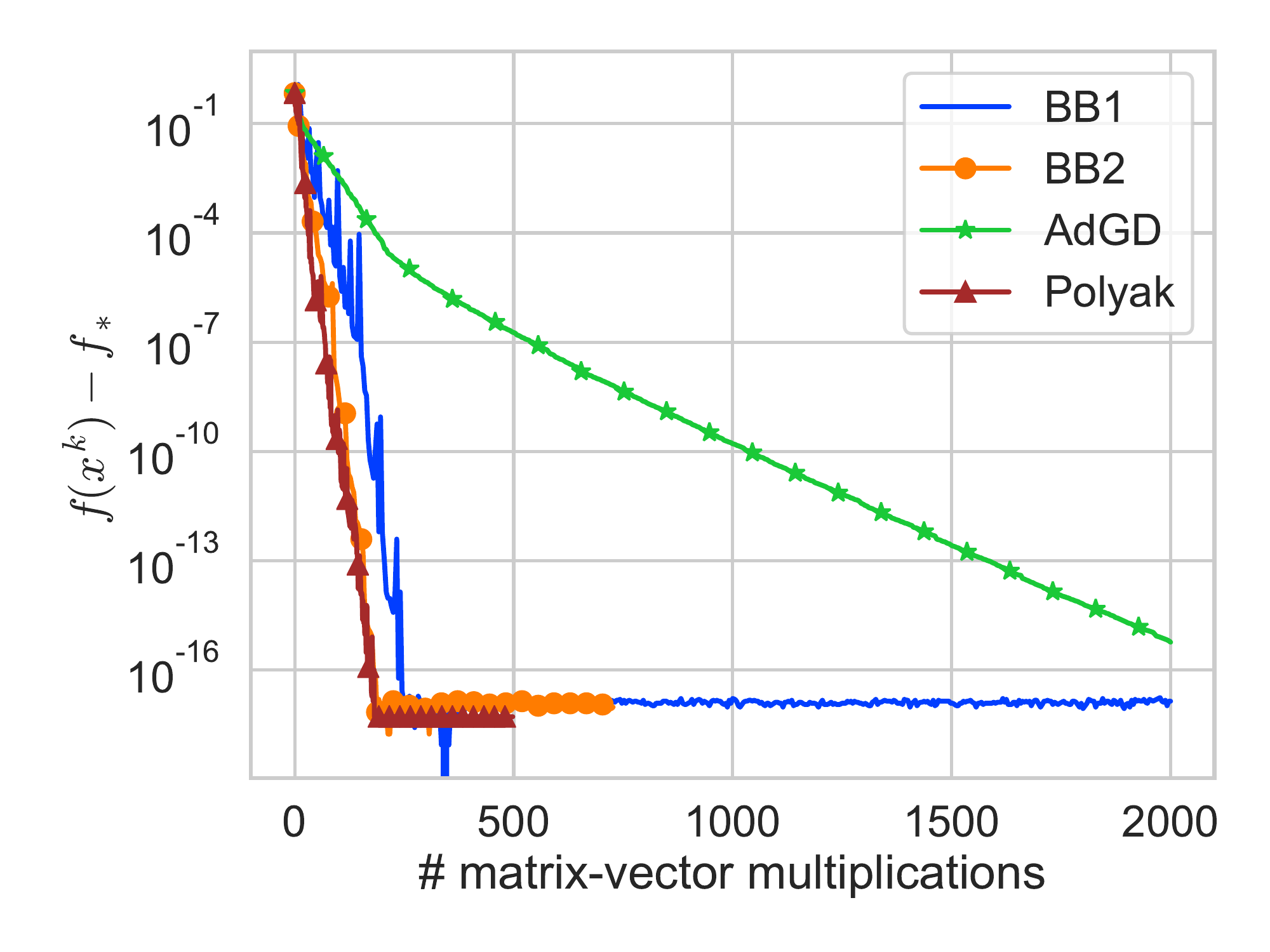}}
\caption{Mushrooms dataset, objective}
\end{subfigure}
\hfill
\begin{subfigure}[t]{0.32\textwidth}
{\includegraphics[trim={3mm 0 3mm 0},clip,width=1\textwidth]{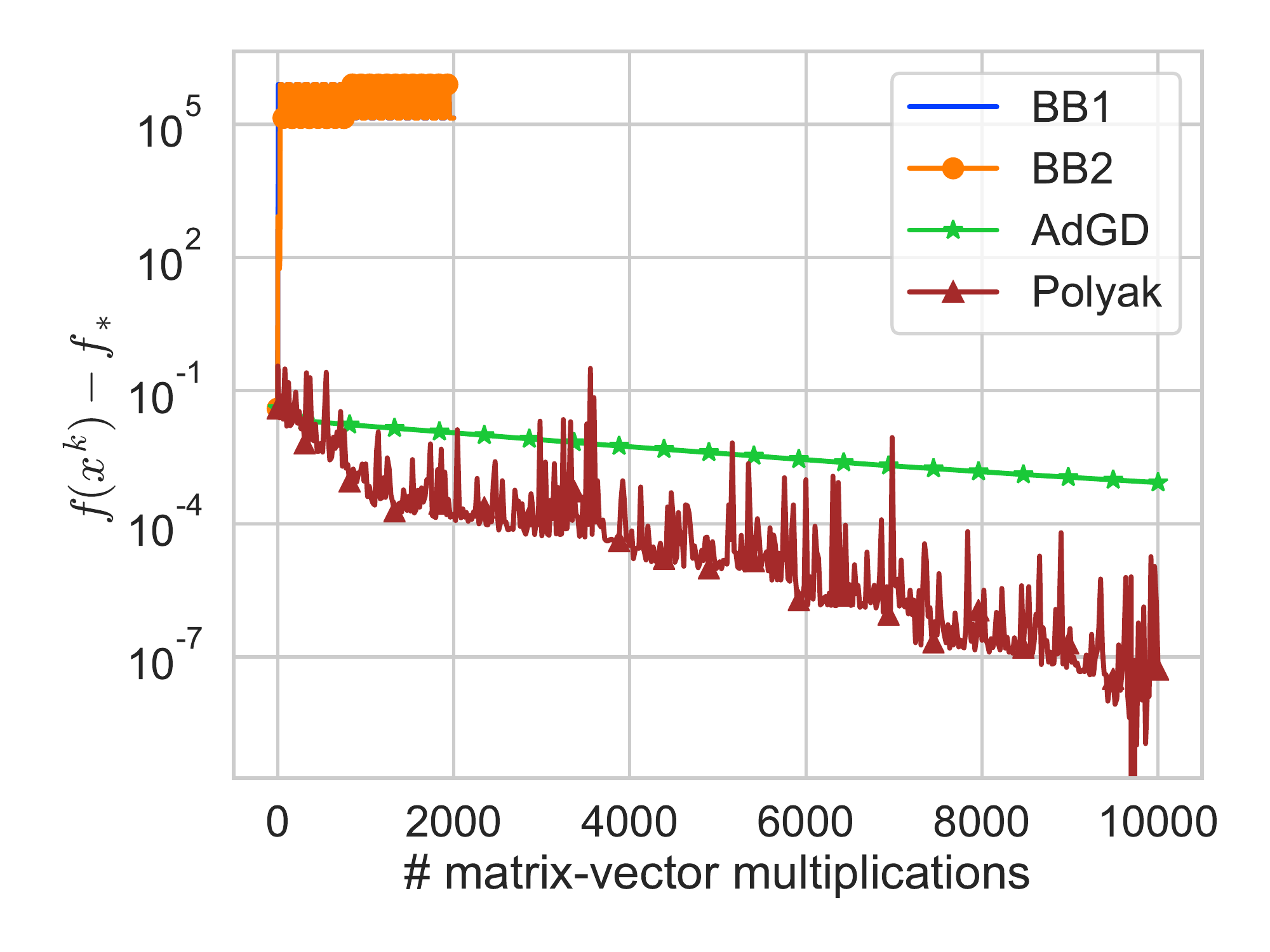}}
\caption{Covtype dataset, objective}
\end{subfigure}
\hfill
 \begin{subfigure}[t]{0.32\textwidth}
{\includegraphics[trim={3mm 0 3mm 0},clip,width=1\textwidth]{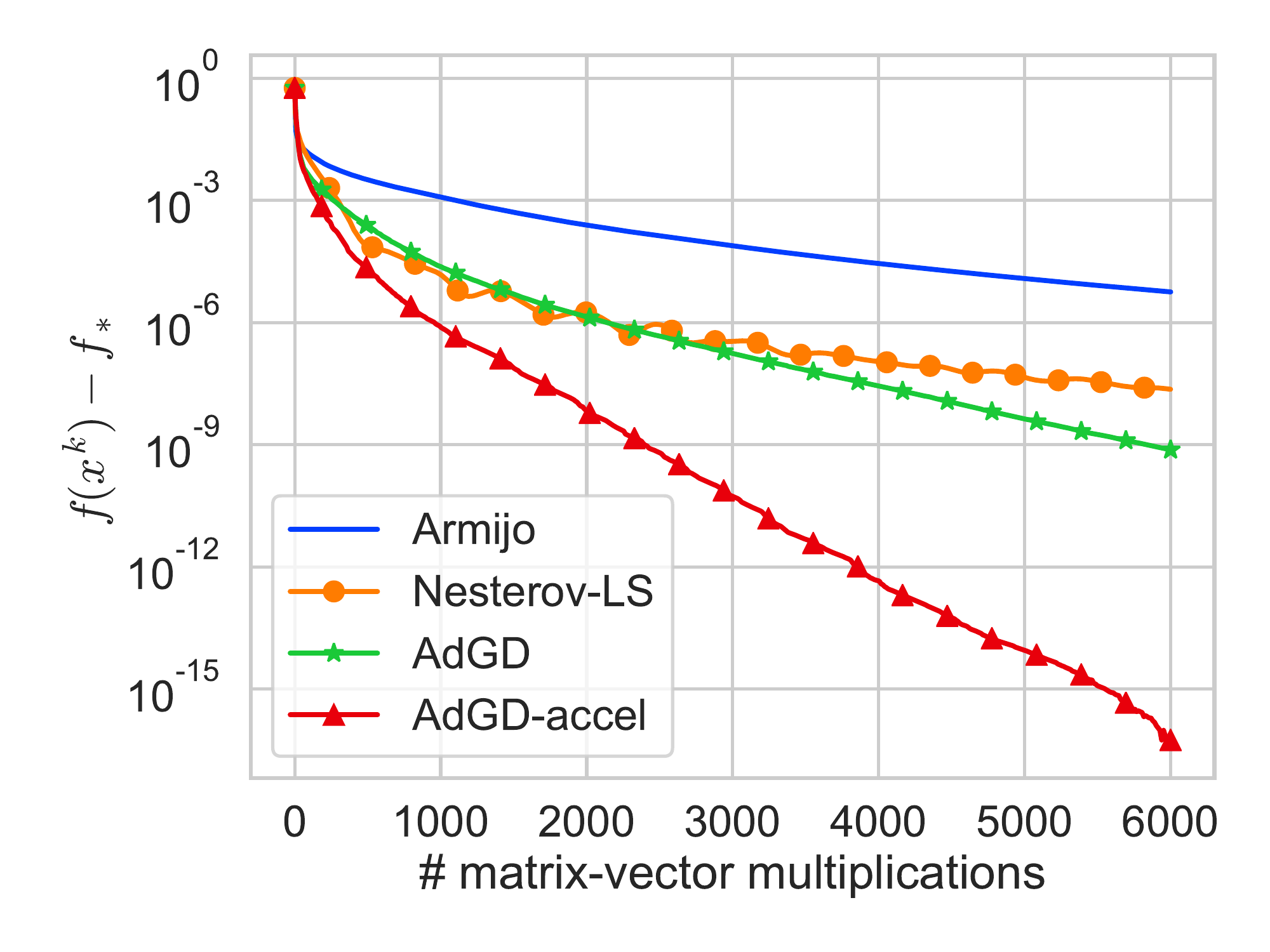}}
 \caption{W8a dataset, objective}
 \end{subfigure}
\caption{Additional results for the logistic regression problem.}
\label{fig:logistic_extra}
\end{figure*}

We have started this paper with an overview of different approaches to
tackle the issue of a stepsize for GD. Now, we
demonstrate some of those solutions. We again consider the
$\ell_2$-regularized logistic regression (same setting as before) with
`mushrooms', `covtype', and `w8a' datasets.

In~\Cref{fig:logistic_extra} (left) we see that the Barzilai-Borwein
method can indeed be very fast.  However, as we said before, it lacks
a theoretical basis and \Cref{fig:logistic_extra} (middle) illustrates
this quite well. Just changing one dataset to another makes both
versions of this method to diverge on a strongly convex and smooth
problem. Polyak's method consistently performs well (see
\Cref{fig:logistic_extra} (left and middle)), however, only after it
was fed with $f_*$ that we found by running another method. Unfortunately,
for logistic regression there is no way to guess this value
beforehand.

Finally, line search for GD (Armijo version) and Nesterov GD
(implemented as in \cite{Nesterov2013a}) eliminates the need to know
the stepsize, but this comes with a higher price per iteration as
\Cref{fig:logistic_extra} (right) shows. Actually in all our
experiments for logistic regression with different datasets one
iteration of Armijo line search was approximately 2 times more
expensive than AdGD, while line search for Nesterov GD was 4 times
more expensive. We note that these observations are consistent with the
theoretical derivations in \cite{Nesterov2013a}.

\paragraph{Neural networks.}
\begin{figure*}[t]
\centering
\begin{subfigure}[t]{0.32\textwidth}
    {\includegraphics[trim={3mm 0 3mm 0},clip,width=1\textwidth]{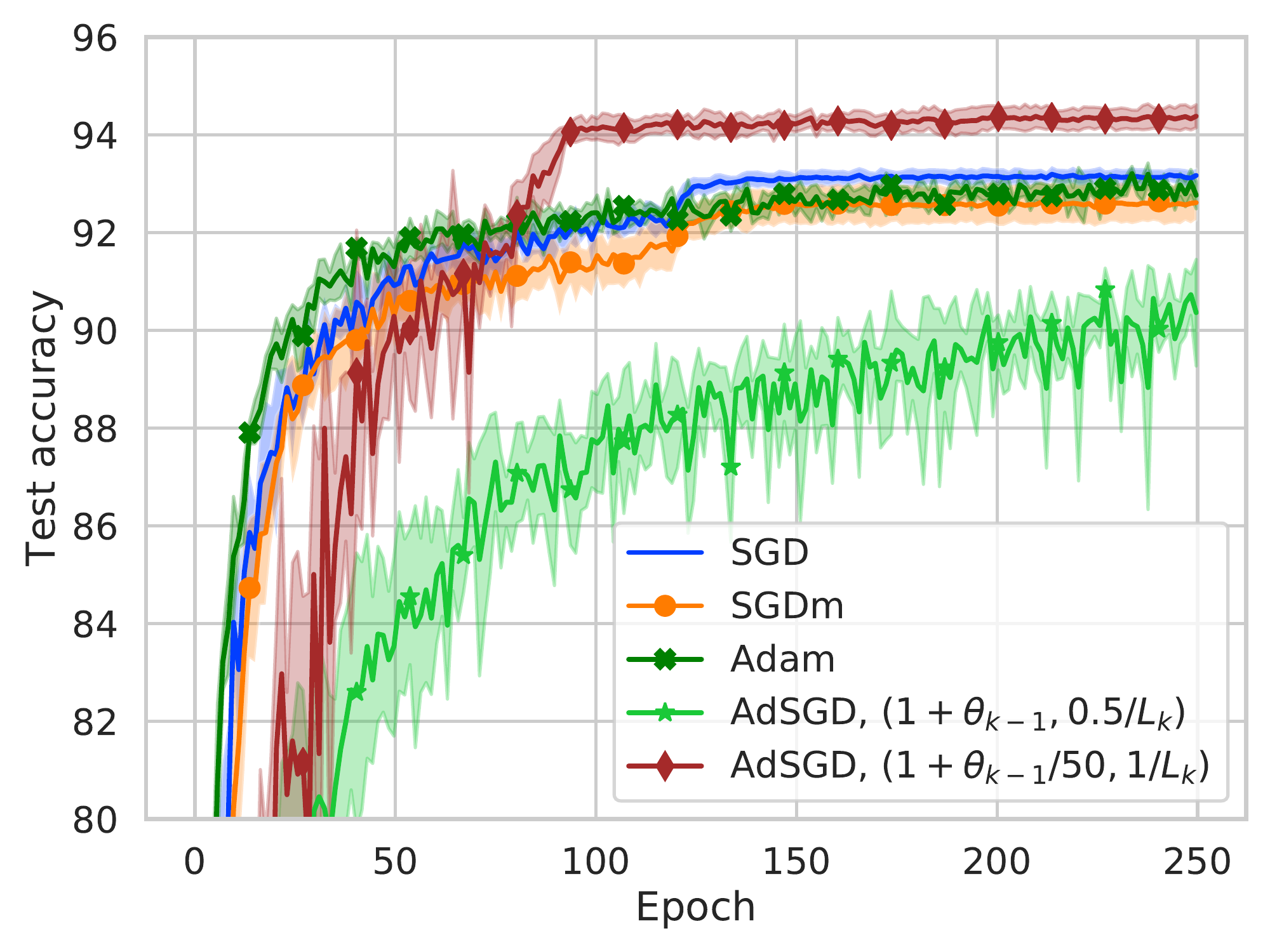}}
\caption{Test accuracy}
\end{subfigure}
\hfill
\begin{subfigure}[t]{0.32\textwidth}
{\includegraphics[trim={3mm 0 3mm 0},clip,width=1\textwidth]{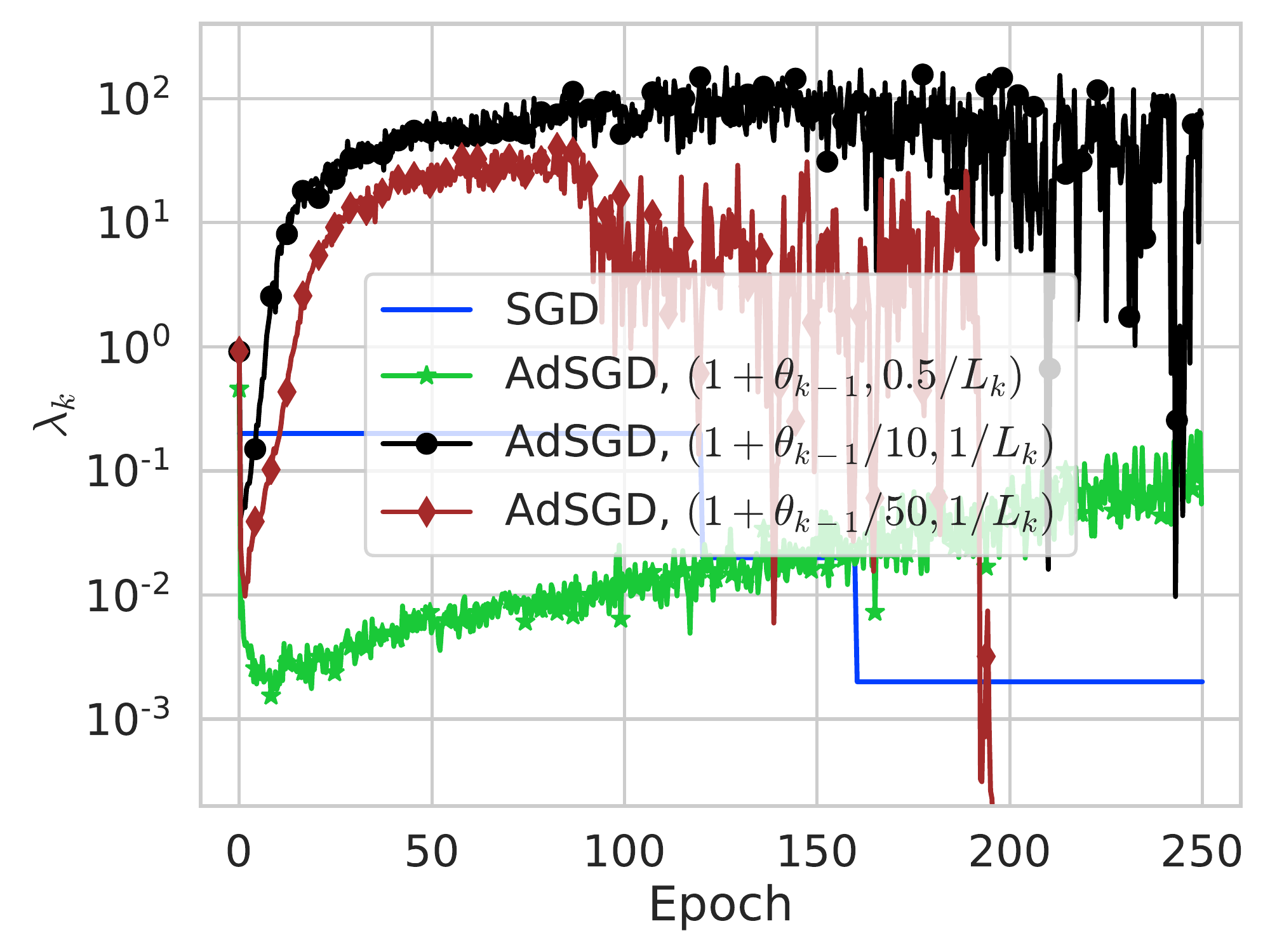}}
\caption{Stepsize}
\end{subfigure}
\hfill
\begin{subfigure}[t]{0.32\textwidth}
{\includegraphics[trim={3mm 0 3mm 0},clip,width=1\textwidth]{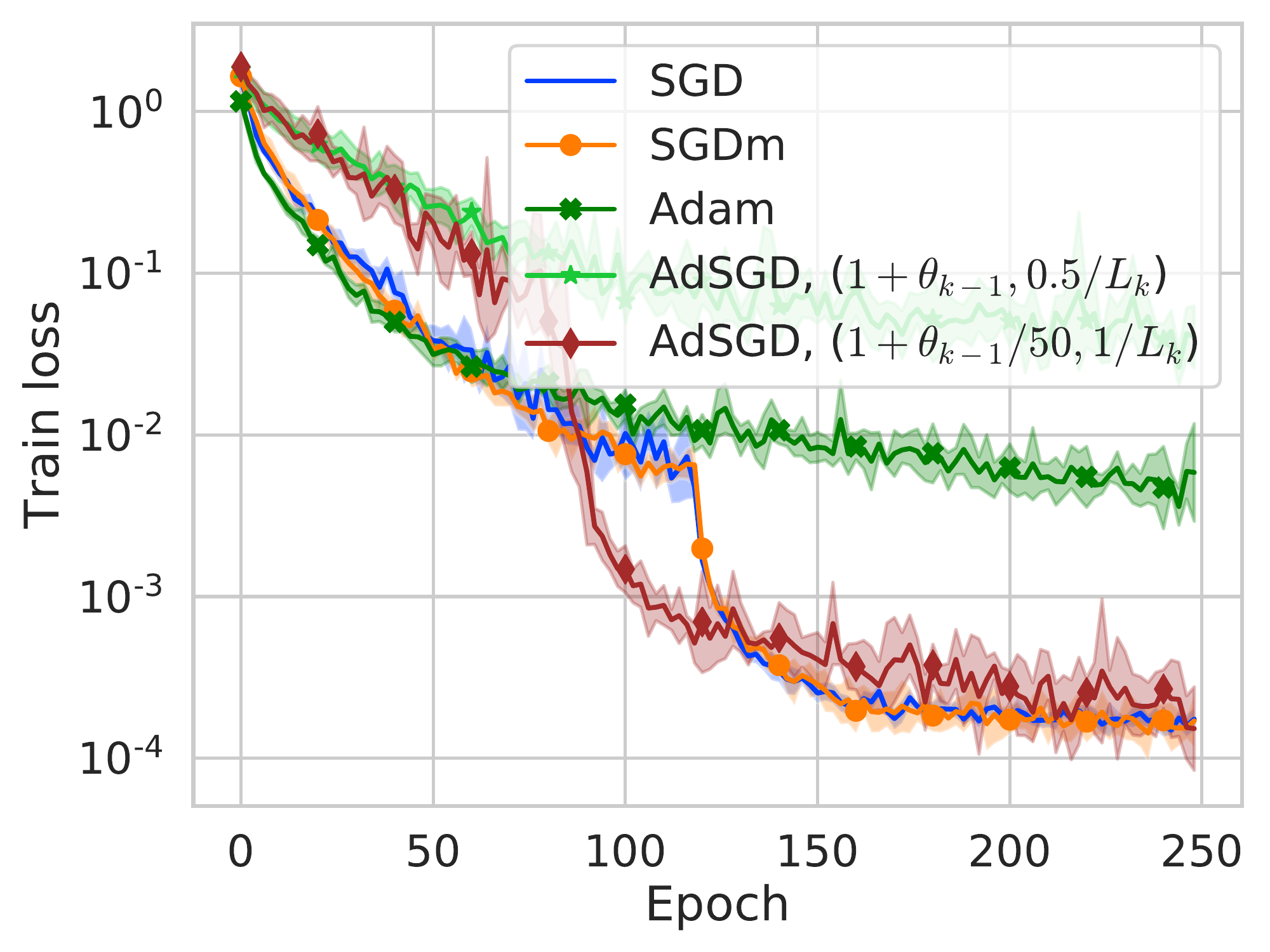}}
\caption{Train loss}
\end{subfigure}
\caption{Results for training ResNet-18 on Cifar10. Labels for AdGD correspond to how $\la_k$ was estimated.}
\label{fig:resnet}
\end{figure*}

\begin{figure*}[t]
\centering
\begin{subfigure}[t]{0.32\textwidth}
    {\includegraphics[trim={3mm 0 3mm 0},clip,width=1\textwidth]{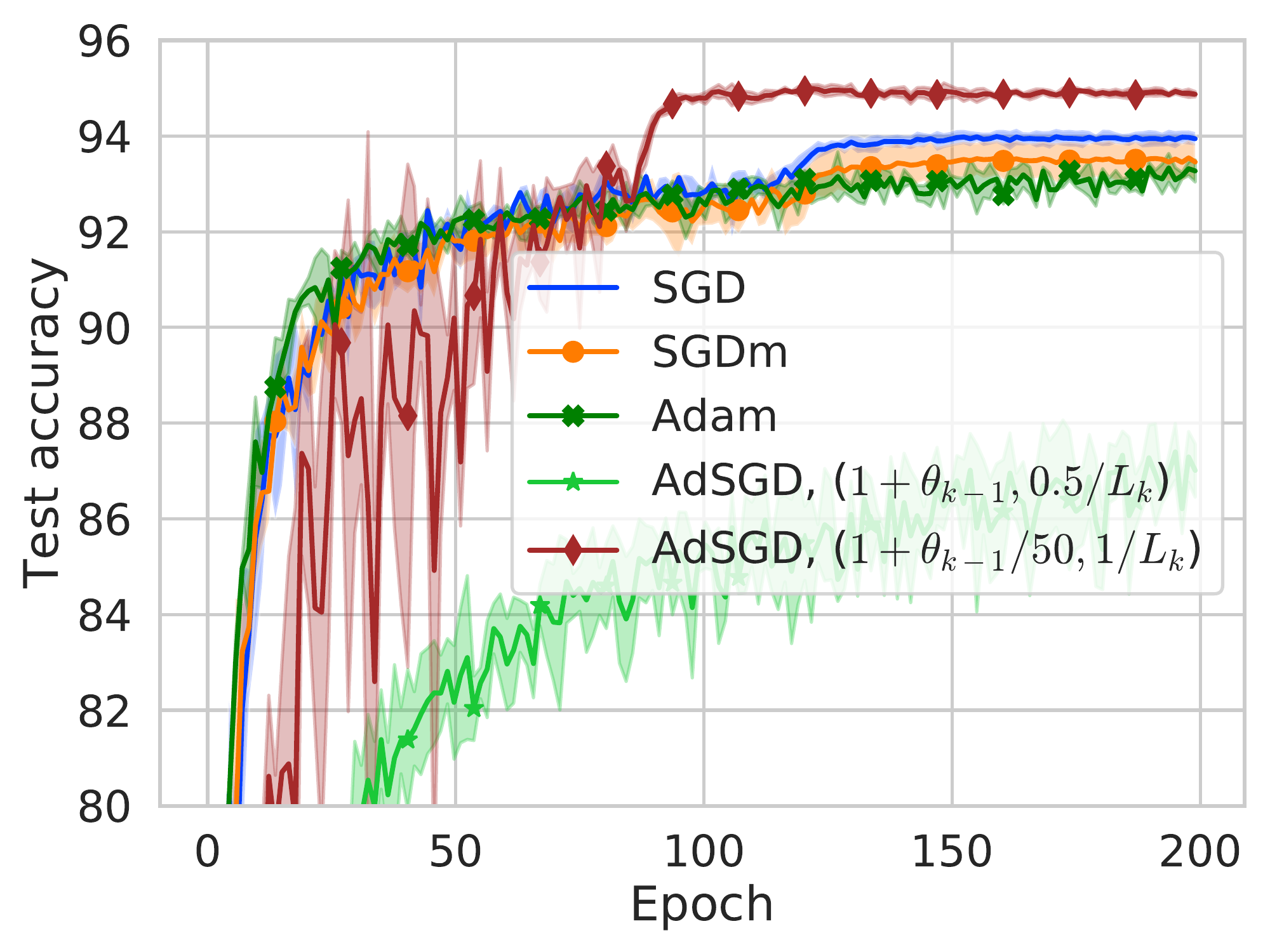}}
\caption{Test accuracy}
\end{subfigure}
\hfill
\begin{subfigure}[t]{0.32\textwidth}
{\includegraphics[trim={3mm 0 3mm 0},clip,width=1\textwidth]{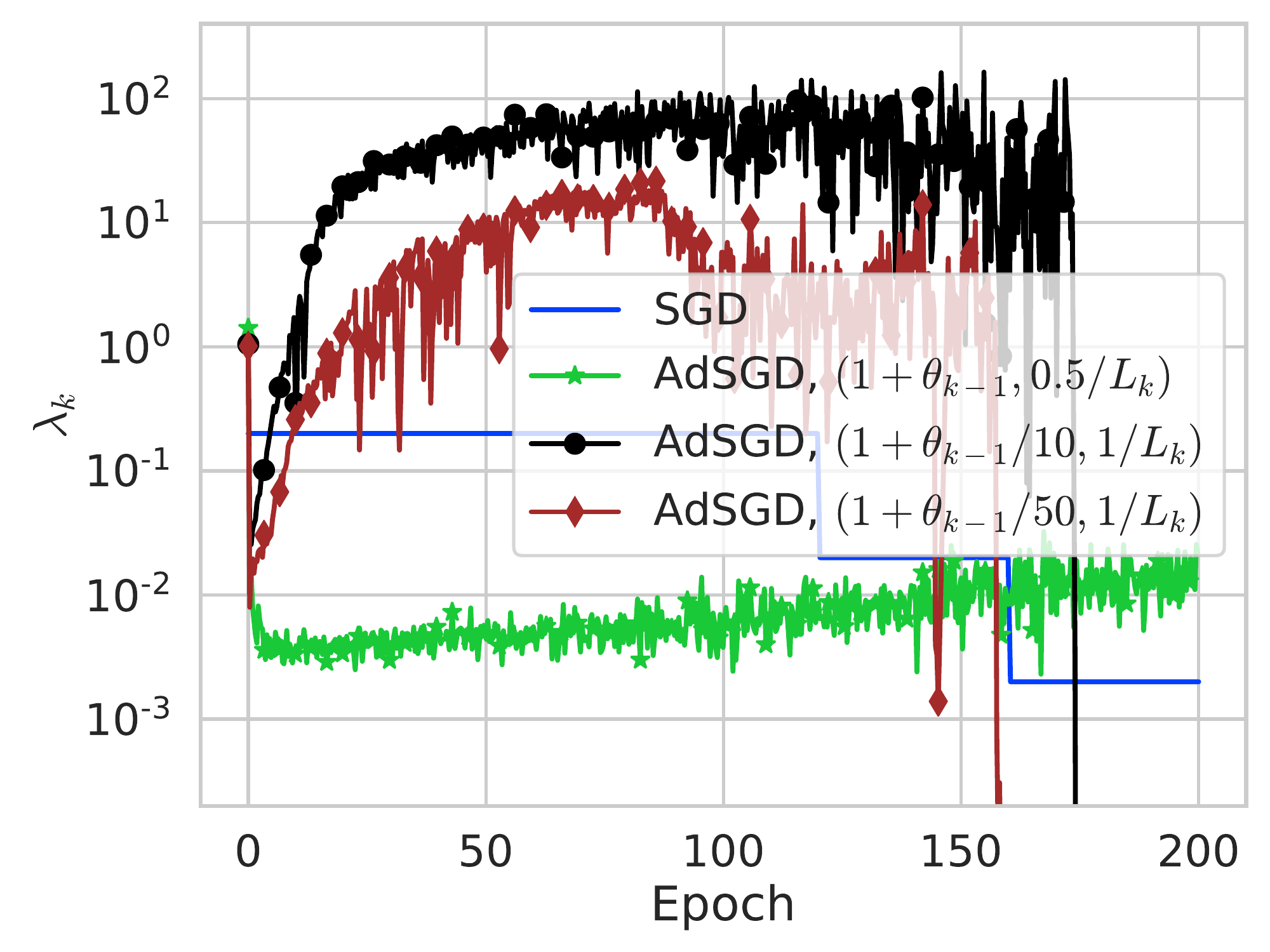}}
\caption{Stepsize}
\end{subfigure}
\hfill
\begin{subfigure}[t]{0.32\textwidth}
{\includegraphics[trim={3mm 0 3mm 0},clip,width=1\textwidth]{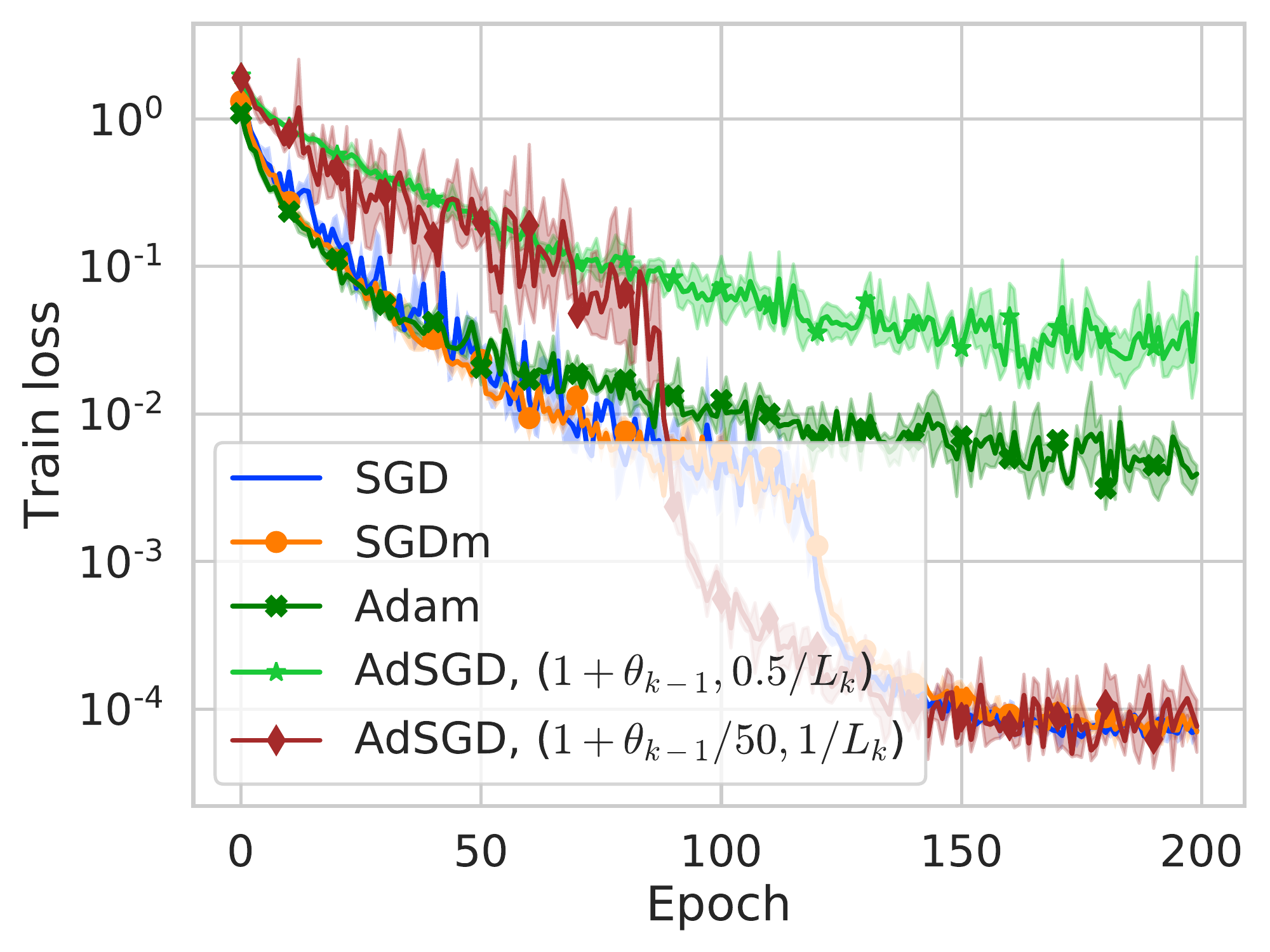}}
\caption{Train loss}
\end{subfigure}
\caption{Results for training DenseNet-121 on Cifar10.}
\label{fig:densenet}
\end{figure*}

We use standard ResNet-18 and DenseNet-121 architectures implemented in PyTorch~\cite{paszke2017automatic} and train them to classify images from the Cifar10 dataset~\cite{krizhevsky2009learning} with cross-entropy loss.

We use batch size 128 for all methods. For our method, we observed
that $\frac{1}{L_k}$ works better than $\frac{1}{2L_k}$. We ran it
with $\sqrt{1+\gamma\th_k}$ in the other factor with values of
$\gamma$ from $\{1, 0.1, 0.05, 0.02, \\ 0.01\}$ and $\gamma=0.02$ performed the best. For reference, we provide the result for the theoretical estimate as well and value $\gamma=0.1$ in the plot with estimated stepsizes.  The results are depicted in~\Cref{fig:resnet,fig:densenet} and other details are provided in~\cref{ap:exp_details}.

We can see that, surprisingly, our method achieves better test accuracy than SGD despite having the same train loss. At the same time, our method is significantly slower at the early stage and the results are quite noisy for the first 75 epochs. Another observation is that the smoothness estimates are very non-uniform and $\lambda_k$ plummets once train loss becomes small.

\section{Perspectives}
We briefly provide a few directions which we personally consider to be
important and challenging.
\begin{enumerate}
    \item \textbf{Nonconvex case.}  A great challenge for us is to
    obtain theoretical guarantees of the proposed method in the
    nonconvex settings. We are not aware of any generic first-order
    method for nonconvex optimization that does not rely on the
    descent lemma (or its
    generalization), see, e.g., \cite{attouch2013convergence}.

    \item \textbf{Performance estimation.} In our experiments we often
    observed much better performance of Algorithm~\ref{alg:main}, than
    GD or AGD. However, the theoretical rate we can show coincides
    with that of GD. The challenge here is to bridge this gap and we hope
    that the approach pioneered by~\cite{drori2014performance} and
    further developed
    in~\cite{taylor2017smooth,kim2016optimized,taylor19a} has a
    potential to do that.

    \item \textbf{Composite minimization.} In classical first-order
    methods, the transition from smooth to composite
    minimization~\cite{Nesterov2013a} is rather
    straightforward. Unfortunately, the proposed proof of
    \Cref{alg:main} does not seem to provide any route for
    generalization and we hope there is some way of resolving this
    issue.

    \item \textbf{Stochastic optimization.} The derived bounds for the
    stochastic case are not satisfactory and have a suboptimal
    dependency on $\kappa$. However, it is not clear to us whether one
    can extend the techniques from the deterministic analysis to
    improve the rate.
    \item \textbf{Heuristics.} Finally, we want to have some solid
    ground in understanding the performance of the proposed heuristics.
\end{enumerate}

\paragraph{Acknowledgment.} Yura Malitsky wishes to thank Roman
Cheplyaka for his interest in optimization that partly inspired the
current work. Yura Malitsky was supported by the ONRG project
N62909-17-1-2111 and HASLER project N16066.

\newpage
\printbibliography
\clearpage
\part*{Appendix: }
\section{Missing proofs}
Recall that in the proof of \Cref{th:main} we only showed boundedness of the iterates and complexity for minimizing $f(x)$. It remains to show that sequence $(x^k)$ converges to a solution.
For this, we  need some variation of the Opial lemma.
\begin{lemma}\label{opial-like}
    Let  $(x^k)$ and $(a_k)$ be two sequences in
    $\R^d$ and $\R_+$ respectively. Suppose that $(x^k)$ is bounded,
    its cluster points belong to $\cX \subset \R^d$ and it
    also holds that
\begin{equation}\label{fejer}
    \n{x^{k+1}-x}^2 + a_{k+1}\leq \n{x^k-x}^2 + a_k, \qquad \forall
    x\in \cX.
\end{equation}
Then $(x^k)$ converges to some element in $\cX$.
\end{lemma}
\begin{proof}
Let $\x_1$, $\x_2$ be any cluster points of $(x^k)$. Thus, there
    exist two subsequences $(x^{k_i})$ and $(x^{k_j})$ such that
    $x^{k_i}\to \x_1$ and $x^{k_j}\to \x_2$.
    Since $\n{x^k-x}^2 + a_k$ is nonnegative and bounded,  $\lim_{k\to \infty}(\n{x^k-x}^2 + a_k)$ exists for any
    $x\in \cX$. Let $x=\x_1$. This yields
\begin{align*}
    \lim_{k\to \infty}(\n{x^k-\x_1}^2 + a_k)&=\lim_{i\to
            \infty}(\n{x^{k_i}-\x_1}^2 + a_{k_i})=\lim_{i\to
                                            \infty}a_{k_i}\\
  & = \lim_{j\to
            \infty}(\n{x^{k_j}-\x_1}^2 + a_{k_j})=\n{\x_2-\x_1}^2 + \lim_{j\to
                                            \infty}a_{k_j}.
\end{align*}
Hence, $\lim_{i\to \infty} a_{k_i} = \lim_{j\to \infty}a_{k_j} +
\n{\x_1-\x_2}^2$. Doing the same with $x=\x_2$ instead of $x=\x_1$,
yields $\lim_{j\to \infty} a_{k_j} = \lim_{i\to \infty}a_{k_i} +
\n{\x_1-\x_2}^2$. Thus, we obtain that $\x_1=\x_2$, which
finishes the proof.
\end{proof}
Another statement that we need here is the following tightening of the convexity property.
\begin{lemma}[Theorem 2.1.5, \cite{Nesterov2013}]
    \label{lemma:coco}
    Let $\mathcal{C}$ be a closed convex set in  $\R^d$. If $f\colon \mathcal{C}\to \R $
    is convex and $L$-smooth, then $\forall x,y\in \mathcal{C}$ it holds
    \begin{align}\label{eq:smooth_and_convex}
    f(x)-f(y)-\lr{\nabla f(y),x-y}\geq \frac{1}{2L}\n{\nabla
            f(x)-\nabla f(y)}^2.
    \end{align}
\end{lemma}

\begin{proof}[\textbf{Proof of
    Theorem~\ref{th:main}}](\textit{Convergence of $(x^k)$})

Note that in the first part we have already proved that $(x^k)$ is bounded and that
$\nabla f$ is $L$-Lipschitz on  $\mathcal{C}=\clconv\{x^*, x^0, x^1, \dots\}$. Invoking Lemma~\ref{lemma:coco}, we deduce that
\begin{equation}
    \label{eq:better_with_Lip}
    \la_k(f(x^*) - f(x^k)) \geq  \la_k \lr{\nabla f(x^k), x^*-x^{k}} +
    \frac{\la_k}{2L}\n{\nabla f(x^k)}^2.
\end{equation}
This indicates that instead of using
inequality~\eqref{eq:conv} in the proof of
Lemma~\ref{lemma:energy}, we could use a better
estimate~\eqref{eq:better_with_Lip}. However, we want to emphasize
that we did not assume that $\nabla f$ is globally Lipschitz, but rather
obtained Lipschitzness on $\mathcal{C}$ as an artifact of our
analysis. Clearly, in the end this improvement gives us an additional term
$\frac{\la_k}{L}\n{\nabla f(x^k)}^2$ in the left-hand side of
\eqref{eq:lemma_ineq}, that is
\begin{multline}
  \label{eq:lemma_ineq_appendix}
\n{x^{k+1}-x^*}^2+ \frac 1 2 \n{x^{k+1}-x^k}^2  + 2\la_{k}(1+\th_{k})
 (f(x^k)-f_*) + \frac{\la_k}{L}\n{\nabla f(x^k)}^2 \\ \leq \n{x^k-x^*}^2  + \frac 1 2
 \n{x^k-x^{k-1}}^2    + 2\la_k \th_k (f(x^{k-1})-f_*).
\end{multline}
Thus, telescoping \eqref{eq:lemma_ineq_appendix}, one
obtains that $\sum_{i=1}^k\frac{\la_k}{L}\n{\nabla f(x^k)}^2\leq D$.
As $\la_k\geq \frac{1}{2L}$, one has that $\nabla f(x^k)\to 0$. Now we
might conclude that all cluster points of $(x^k)$ are solutions of~\eqref{main}.

Let $\cX$ be the solution set of \eqref{main} and
$a_k = \frac 1 2 \n{x^k-x^{k-1}}^2 + 2\la_k \th_k(f(x^{k-1})-f_*)$.
We want to finish the proof applying Lemma~\ref{opial-like}. To
this end, notice that inequality~\eqref{eq:lemma_ineq}
 yields \eqref{fejer}, since $\la_{k+1}\th_{k+1}\leq
 (1+\th_k)\la_k$. This completes the proof.
\end{proof}

\begin{proof}[\textbf{Proof of \Cref{th:strong}}]~\\
    First of all, we note that using the stricter inequality
    $\la_k\leq \sqrt{1+\frac{\th_{k-1}}{2}}\la_{k-1}$ does not change
    the statement of \Cref{th:main}. Hence, $x^k\to x^*$ and there
    exist $\mu,L >0$ such that $f$ is $\mu$-strongly convex and
    $\nabla f$ is $L$-Lipschitz on
    $\mathcal{C}=\clconv\{x^*, x^0, x^1, \dots\}$. Secondly, due to
    local strong convexity,
    $\n{\nabla f(x^k)-\nabla f(x^{k-1})}\geq \mu \n{x^k-x^{k-1}}$, and
    hence $\la_k \leq \frac{1}{2\mu}$ for $k\geq 1$.

    Now we tighten some steps in the analysis to
    improve bound~\eqref{eq:conv}. By strong convexity,
\begin{align*}
   \la_k \lr{\nabla f(x^k), x^*-x^{k}} & \leq \la_k(f(x^*) - f(x^k)) - \la_k\frac{\mu}{2}\|x^* - x^k\|^2.
\end{align*}
By $L$-smoothness and bound $\la_k\le \frac{1}{2\mu}$,
\begin{align*}
   \la_k \lr{\nabla f(x^k), x^*-x^{k}} & \leq \la_k(f(x^*) - f(x^k)) - \la_k\frac{1}{2L}\|\nabla f(x^k)\|^2 \\
   &= \la_k(f_* - f(x^k)) - \frac{1}{2L\la_k}\|x^{k+1}-x^k\|^2 \\
   &\le \la_k(f_* - f(x^k)) - \frac{\mu}{L}\|x^{k+1}-x^k\|^2.
\end{align*}
Together, these two bounds give us
\begin{align*}
   \la_k \lr{\nabla f(x^k), x^*-x^{k}} & \leq \la_k(f_* - f(x^k)) -
                                         \la_k\frac{\mu}{4}\|x^k - x^*\|^2 - \frac{\mu}{2L}\|x^{k+1}-x^k\|^2.
\end{align*}

We keep inequality~\eqref{eq:terrible_simple} and the rest of the proof as is.
Then the strengthen analog of~\eqref{eq:lemma_ineq} will be
\begin{align}
  &\n{x^{k+1}-x^*}^2+ \frac 1 2\left(1 + \frac{2\mu}{L}\right) \n{x^{k+1}-x^k}^2 + 2\la_k
  (1+\th_k) (f(x^k)-f_*) \nonumber \\
  \leq \, & \left(1 - \frac{\la_k\mu}{2}\right)\n{x^k-x^*}^2  + \frac 1 2
            \n{x^k-x^{k-1}}^2   + 2\la_k \th_k (f(x^{k-1})-f_*) \nonumber\\
   \leq \, & \left(1 - \frac{\la_k\mu}{2}\right)\n{x^k-x^*}^2  + \frac 1 2
 \n{x^k-x^{k-1}}^2   + 2\la_{k-1} \left(1+\frac{\th_{k-1}}{2}\right) (f(x^{k-1})-f_*),
            \label{eq:contraction}
\end{align}
where in the last inequality we used our new condition on $\la_k$.
Under the new update we have contraction in every term:
$1-\frac{\la_k\mu}{2}$ in the first,
$\frac{1}{1+2\mu/L}=1-\frac{2\mu}{L+2\mu}$ in the second and
$\frac{1+\th_{k-1}/2}{1+\th_{k-1}}=1-\frac{\th_{k-1}}{2(1+\th_{k-1})}$
in the last one.

To further bound the last contraction, recall that
$\la_k \in \left[\frac{1}{2L}, \frac{1}{2\mu}\right]$ for $k\ge 1$.
Therefore, $\th_k = \frac{\la_k}{\la_{k-1}} \ge \frac{1}{\kappa}$ for any $k>1$, where
$\kappa\eqdef \frac{L}{\mu}$. Since the function $\th\mapsto \frac{\th}{1+\th}$
monotonically increases with $\th>0$,  this implies
$\frac{\th_{k-1}}{2(1+\th_{k-1})}\ge \frac{1}{2(\kappa+1)}$ when
$k>2$. Thus, for the full energy $\Psi^{k+1}$ (the left-hand side
of~\eqref{eq:contraction}) we have
\begin{align*}
        \Psi^{k+1}
        &\le \left(1  - \min\left\{\frac{\la_k \mu}{2}, \frac{1}{2(\kappa+1)}, \frac{2\mu}{L+2\mu}\right\}\right) \Psi^k.
\end{align*}
Using simple bounds  $\frac{\la_k\mu}{2}\geq \frac{1}{4\kappa}$,
$\frac{2\mu}{L+2\mu}=\frac{2}{\kappa +2}\geq \frac{1}{4\kappa}$, and
$\frac{1}{2(\kappa +1)}\geq \frac{1}{4\kappa}$, we obtain $\Psi^{k+1}\le (1 - \frac{1}{4\kappa})\Psi^k$ for $k>2$.
This gives $\mathcal{O}\left(\kappa \log \frac{1}{\varepsilon}\right)$ convergence rate.
\end{proof}
\section{Extensions}
\subsection{More general update}
One may wonder how flexible the update for $\la_k$ in \Cref{alg:main}
is? For example, is it necessary to upper bound the stepsize with
$\sqrt{1+\th_k}\la_{k-1}$ and put $2$ in the denominator of
$\frac{\n{x^k-x^{k-1}}}{2\n{\nabla f(x^k)-\nabla f(x^{k-1})}}$?
\Cref{alg:general_update} that we present here  partially answers
 this question.
 \begin{algorithm}[t]
 \caption{Adaptive gradient descent (general update)}
 \label{alg:general_update}
 \begin{algorithmic}[1]
     \STATE \textbf{Input:} $x^0 \in \R^d$, $\la_0>0$,
     $\th_0=+\infty$, $\a\in (0,1)$, $\b = \frac{1}{2(1-\a)}$\\
     \STATE  $x^1= x^0-\la_0\nabla f(x^0)$
        \FOR{$k = 1,2,\dots$}
        \STATE $\la_k = \min\Bigl\{
        \sqrt{\frac{1}{\beta}+\th_{k-1}}\la_{k-1},\frac{\a\n{x^{k}-x^{k-1}}}{\n{\nabla
        f(x^{k})-\nabla f(x^{k-1})}}\Bigr\}$
\STATE $x^{k+1} = x^k - \la_k \nabla f(x^k)$
\STATE $\th_k = \frac{\la_k}{\la_{k-1}}$
        \ENDFOR
 \end{algorithmic}
 \end{algorithm}
Obviously, \Cref{alg:main} is a particular case of
\Cref{alg:general_update} with $\a=\frac 12$ and $\b =1$.

\begin{theorem}\label{th:gen_update}
    Suppose that $f\colon \R^d\to \R$ is convex with locally Lipschitz
    gradient $\nabla f$. Then $(x^k)$ generated by~\Cref{alg:general_update} converges to a solution of \eqref{main}
    and we have that
\[f(\hat
    x^k)-f_* \leq \frac{D}{2S_k}=\mathcal{O}\Bigl(\frac{1}{k}\Bigr),\]
where
\begin{align*}
\hat x^k &=  \frac{\la_k(1+\th_k\b)x^k +
        \sum_{i=1}^{k-1}w_i
           x^i}{S_k},\\
    w_i &= \la_i(1+\th_i\b)-\la_{i+1}\th_{i+1}\b,\\
  S_k &= \la_k(1+\th_k\b) + \sum_{i=1}^{k-1}w_i = \sum_{i=1}^k \la_i + \la_1\th_1\b,
\end{align*}
and $D$ is a constant
that explicitly depends on the initial data and the solution set.
\end{theorem}
\begin{proof}
    Let $x^*$ be arbitrary solution of \eqref{main}.  We note
    that equations~\eqref{eq:norms} and~\eqref{dif_x} hold for any
    variant of GD, independently of $\la_k$, $\a$, $\b$.  With the new
    rule for  $\la_k$, from \eqref{dif_x} it follows
\begin{align*}
        \|x^{k+1}-x^k\|^2
        &\le 2\la_k\th_k(f(x^{k-1})-f(x^k))-\|x^{k+1}-x^k\|^2+2\la_k\|\nabla f(x^k)-\nabla f(x^{k-1})\|\| x^k-x^{k+1}\| \\
        &\le 2\la_k\th_k(f(x^{k-1})-f(x^k))-(1-\alpha)\|x^{k+1}-x^k\|^2+\alpha\|x^k-x^{k-1}\|^2,
\end{align*}
which, after multiplication by $\beta$ and reshuffling the terms, becomes
\begin{align*}
        \beta(2 - \alpha)\|x^{k+1}-x^k\|^2+2\beta\la_k\th_k (f(x^k)-f_*)
        \le \alpha\beta\|x^k-x^{k-1}\|^2+ 2\beta\la_k\th_k (f(x^{k-1})-f_*).
\end{align*}
Adding \eqref{eq:norms} and the latter inequality gives us
\begin{align*}
        &\|x^{k+1}-x^*\|^2 + 2\la_k(1+\th_k\beta)(f(x^k)-f_*) + (2\beta-\alpha\beta-1)\|x^{k+1}-x^k\|^2\\
        &\le \|x^k-x^*\|^2+2\la_k\th_k\beta(f(x^{k-1})-f_*) +\alpha\beta\|x^{k+1}-x^k\|^2.
\end{align*}
Notice that by $\b = \frac{1}{2(1-\a)}$, we have $2\b-\a\b -1= \a\b$ and
hence,
\begin{align*}
        &\|x^{k+1}-x^*\|^2 + 2\la_k(1+\th_k\beta)(f(x^k)-f_*) +\a\b\|x^{k+1}-x^k\|^2\\
        &\le \|x^k-x^*\|^2+2\la_k\th_k\beta(f(x^{k-1})-f_*) +\alpha\beta\|x^{k+1}-x^k\|^2.
\end{align*}
As a sanity check, we can see that with $\a=\frac{1}{2}$ and $\b
=1$, the above inequality coincides with~\eqref{eq:lemma_ineq}.

Telescoping this inequality, we deduce
\begin{align}\label{eq:telescope2}
        \n{x^{k+1}-x^*}^2&+ \a\b \n{x^{k+1}-x^k}^2  + 2\la_k
  (1+\th_k\b) (f(x^k)-f_*) \notag \\ &\qquad +
  2\sum_{i=1}^{k-1}[\la_i(1+\th_i\b)-\la_{i+1}\th_{i+1}\b](f(x^i)-f_*) \notag\\
  &\leq  \n{x^1-x^*}^2  + \a\b \n{x^1-x^{0}}^2   + 2\la_1 \th_1\b [f(x^{0})-f_*]\eqdef D.
\end{align}
Note that because of the way we defined stepsize,
$\la_i(1+\th_i\b)-\la_{i+1}\th_{i+1}\b\geq 0$. Thus, the sequence
$(x^k)$ is bounded. Since $\nabla f$ is locally Lipschitz, it is
Lipschitz continuous on bounded sets. Let $L$ be a Lipschitz constant
of $\nabla f$ on a bounded set
$\mathcal{C}=\clconv\{x^*,x^1,x^2,\dots\}$.

If $\a \leq \frac 1 2 $, then $\frac{1}{\b}>1$ and similarly
to~\Cref{th:main} we might conclude that $\la_k\geq\frac{\a}{L}$ for
all $k$. However, for the case $\a > \frac 1 2$ we cannot do
this. Instead, we prove  that $\la_k\ge \frac{2\a(1-\a)}{L}$, which suffices
 for our purposes.

 Let $m, n\in \N$ be the smallest numbers such that
 $\b^{-\frac{1}{2^m}}\ge 1 - \frac{1}{2\b}$ and
 $(1+\frac{1}{\b})^{\frac n 2}\ge \b$. We want
 to prove that for any $k$ it holds $\la_k\ge \frac{2\a(1-\a)}{L}$ and among every $m+n+1$ consecutive elements
 $\la_k,\la_{k+1},\dots,\la_{k+m+n}$ at least one is no less than
 $\frac{\a}{L}$. We shall prove this by induction. First, note that
 the second bound always satisfies
 $\frac{\a\n{x^k-x^{k-1}}}{\n{\nabla f(x^k)-\nabla f(x^{k-1})}}\geq
 \frac{\a}{L}$ for all $k$,  which also
 implies that $\la_1\geq \frac{\a}{L}$. If for all $k$ we have $\la_k\ge \frac{\a}{L}$, then we are done. Now assume that $\la_{k-1}\geq
 \frac{\a}{L}$ and $\la_k<\frac{\a}{L}$ for some $k$. Choose the largest $j$ (possibly infinite) such that the second bound is not
 active for $\la_k, \dotsc, \la_{k+j-1}$, i.e.,  $\la_{k+i}=\sqrt{\frac{1}{\b} + \th_{k+i-i}}\la_{k+i-1}$ for $i<j$.

Let us prove that $\la_k, \dotsc, \la_{k+j-1}\ge \frac{2\a(1-\a)}{L}$.  The definition of $j$ yields $\th_{k+i}=\sqrt{\frac{1}{\b} +
     \th_{k+i-1}}$ for all
 $i=0,\dots, j-1$. Recall that $\b > 1$, and
 thus,
 \[\th_k\geq  \b^{-\frac 1 2},\quad \th_{k+1}\geq \sqrt{\frac{1}{\b}
         + \sqrt{\frac{1}{\b}}}\geq \b^{-\frac{1}{4}},\quad \dotsc,
     \quad \th_{k+i}\geq \sqrt{\frac{1}{\b} + \b^{-\frac{1}{2^i}}} \geq
    \b^{-\frac{1}{2^{i+1}}}\]
 for all $i<j$. Now it remains to notice that for any $i<j$
 \[\frac{\la_{k+i}}{\la_{k-1}} = \th_{k}\th_{k+1}\dots
 \th_{k+i} \geq \beta^{-\frac{1}{2}-\frac{1}{4}-\dotsb -\frac{1}{2^{i+1}} }\geq \frac{1}{\beta}=2(1-\a)\]
and hence $\la_{k+i}\geq 2(1-\a)\la_{k-1}\geq \frac{2\a(1-\a)}{L}$. If
$j\leq m+n$, then at $(k+j)$-th iteration the second bound is active,
i.e., $\la_{k+j}\ge \frac{\a}{L}$, and we are done with the other claim as well. Otherwise, note
\[
        \th_{k+m-1} \ge \b^{-\frac{1}{2^m}} \ge 1 - \frac{1}{2\b},
\]
 so $\th_{k+m}=\sqrt{\frac{1}{\b}+\th_{k+m-1}}\ge \sqrt{\frac{1}{\b}+1 - \frac{1}{2\b}}=\sqrt{1 + \frac{1}{2\b}}$ and for any $i\in [m, j-2]$ we have $\th_{k+i+1}= \sqrt{\frac{1}{\b}+\th_{k+i}}\ge \sqrt{\frac{1}{\b}+1}$. Thus,
 \[
 \la_{k+m+n} =
 \la_{k-1}\Bigl(\prod_{l=k}^{k+m-1}\th_l\Bigr)\Bigl(\prod_{l=k+m}^{k+m+n}\th_l\Bigr)\ge
 \la_{k-1}\frac{1}{\b}\sqrt{1+\frac{1}{2\b}}\Bigl(1+\frac{1}{\b}\Bigr)^{\frac
 n 2}\ge \la_{k-1} \ge \frac{\a}{L},
 \]
 so we have shown the second claim too.

 To conclude, in both cases
$\a\leq \frac 12$ and $\a>\frac 12$, we have $S_k = \Omega(k)$.

Applying the Jensen inequality for the sum of all terms
$f(x^i)-f_*$ in the left-hand side of~\eqref{eq:telescope2}, we obtain
\[\frac D 2 \geq \frac{\text{LHS of \eqref{eq:telescope2}}}{2} \geq S_k (f(\hat
    x^k)-f_*),\]
where $\hat x^k$ is defined in the statement of the
theorem. Finally, convergence of $(x^k)$ can be proved in a similar
way as \Cref{th:main}.
\end{proof}

\subsection{$f$ is $L$-smooth}
Often, it is known that $f$ is smooth and even some estimate for the
Lipschitz constant $L$ of $\nabla f$ is available. In this case, we
can use slightly larger steps, since instead of just convexity the stronger inequality in \Cref{lemma:coco} holds. To take advantage of it, we present a modified version of \Cref{alg:main} in
\Cref{alg:smooth}. Note that we have chosen to modify \Cref{alg:main} and
not its more general variant~\Cref{alg:general_update} only for simplicity.

\begin{algorithm}[t]
 \caption{Adaptive GD ($L$ is known)}
 \label{alg:smooth}
 \begin{algorithmic}[1]
     \STATE \textbf{Input:} $x^0 \in \R^d$, $\la_0 = \frac{1}{L}$,
     $\th_0=+\infty$\\
     \STATE $x^1 = x^0-\la_0\nabla f(x^0)$
     \FOR{$k = 1,2,\dots$}
        \STATE $L_k = \frac{\n{\nabla
        f(x^{k})-\nabla f(x^{k-1})}}{\n{x^{k}-x^{k-1}}}$
     \STATE  $\la_k = \min\left\{
        \sqrt{1+\th_{k-1}}\la_{k-1},\frac{1}{\la_{k-1}L^2}+\frac{1}{2L_k} \right\}$
\STATE $x^{k+1} = x^k - \la_k \nabla f(x^k)$
\STATE $\th_k = \frac{\la_k}{\la_{k-1}}$
 \ENDFOR
 \end{algorithmic}
 \end{algorithm}
\begin{theorem}\label{theorem:energy-L}
Let $f$ be convex and $L$-smooth. Then for $(x^k)$ generated
by Algorithm~\ref{alg:smooth} inequality~\eqref{eq:lemma_ineq} holds. As a corollary, it holds for some ergodic vector $\hat x^k$ that $f(\hat x^k)-f_*=\mathcal{O}\left(\frac{1}{k}\right)$.
\end{theorem}
\begin{proof}
    Proceeding similarly as in Lemma~\ref{lemma:energy}, we have
  \begin{equation}
      \label{eq:2_simple-L}
          \|x^{k+1}- x^*\|^2 = \|x^k - x^*\|^2 - 2\la_k \<\nabla
            f(x^k), x^k-x^*> + \|x^{k+1} - x^{k}\|^2.
\end{equation}

By convexity of $f$ and Lemma~\ref{lemma:coco},
\begin{align}\label{eq:conv2_simple-L}
   2\la_k \lr{\nabla f(x^k), x^*-x^{k}} & \overset{\eqref{eq:smooth_and_convex}}{\leq} 2\la_k(f(x^*) -
                                         f(x^k)-\frac{1}{2L}\n{\nabla
                                         f(x^k)}^2) \notag \\&= 2\la_k(f_* -
                                         f(x^k))-\frac{1}{\la_k L}\n{x^{k+1}-x^k}^2.
\end{align}
As in~\eqref{dif_x}, we have
 \begin{equation}\label{dif_x-L}
      \|x^{k+1} -x^k\|^2  =  2\la_k \lr{\nabla
                                   f(x^k)-f(x^{k-1}), x^{k}-x^{k+1}} +
                                   2\la_k\lr{f(x^{k-1}), x^k-x^{k+1}}
      -  \n{x^{k+1}-x^k}^2.
\end{equation}
Again, instead of  using merely convexity of $f$, we combine it with
\Cref{lemma:coco}. This gives
    \begin{align}    \label{eq:terrible_simple-L}
          2\la_k\lr{\nabla f(x^{k-1}), x^k-x^{k+1}}
      &=  \frac{2\la_k}{\la_{k-1}}\lr{x^{k-1} - x^{k}, x^{k}-x^{k+1}}\notag
      \\ &=  2\la_k\th_k \lr{x^{k-1}-x^{k}, \nabla
        f(x^k)} \notag \\ & \overset{\eqref{eq:smooth_and_convex}}{\leq}
      2\la_k\th_k
    (f(x^{k-1})-f(x^k)) - \frac{\la_k\th_k}{L}\n{\nabla f(x^k)-\nabla f(x^{k-1})}^2.
\end{align}
Since now we have two additional terms
$\frac{1}{\la_kL}\n{x^{k+1}-x^k}^2$ and $\frac{\la_k\th_k}{L}\n{\nabla
    f(x^{k}) - \nabla f(x^{k-1})}^2$, we can do better than~\eqref{cs}. But first we need a simple, yet a bit tedious fact.
By our choice  of $\la_k$, in every iteration $\la_k\leq \frac{1}{\la_{k-1}L^2} +
\frac{1}{2L_k}$ with $L_k= \frac{\n{\nabla f(x^k)-\nabla f(x^{k-1})}}{\n{x^k-x^{k-1}}}
$. We want to show that it implies
\begin{equation}\label{so_much_trouble}
    2\left(\la_k - \frac{\sqrt{\th_k}}{L} \right)\leq \frac{1}{L_k},
\end{equation}
which is equivalent to $\la_k-\frac{\sqrt{\la_k}}{\sqrt{\la_{k-1}}L}-\frac{1}{2L_k}\le 0$. Nonnegative solutions of the quadratic inequality $t^2 -
\frac{t}{\sqrt{\la_{k-1}}L}-\frac{1}{2L_k}\leq 0$ are
\[0\leq t\leq \frac{1}{2\sqrt{\la_{k-1}}L}+\frac{1}{2}\sqrt{\frac{1}{\la_{k-1}L^2}+\frac{2}{L_k}}
    = \frac{1}{2\sqrt{\la_{k-1}}L}\left(1 + \sqrt{1 + \frac{2\la_{k-1}L^2}{L_k}}\right).\]
Let us prove that $\sqrt{\frac{1}{\la_{k-1}L^2} +
\frac{1}{2L_k}}$ falls into this segment and, hence, $\sqrt{\la_k}$ does as well. Using a simple inequality $4 + a\leq (1 + \sqrt{1+a})^2 $, for
$a>0$, we obtain
\[\frac{1}{\la_{k-1}L^2} + \frac{1}{2L_k} =
    \frac{1}{4\la_{k-1}L^2} \bigl(4 +
    \frac{2\la_{k-1}L^2}{L_k}\bigr)\leq \frac{1}{4{\la_{k-1}}L^2}\left(1 + \sqrt{1 + \frac{2\la_{k-1}L^2}{L_k}}\right)^2.\]
This confirms that~\eqref{so_much_trouble} is true.
Thus, by
Cauchy-Schwarz and Young's inequalities, one has
\begin{equation}\label{cs-L}
\begin{aligned}
  & 2\la_k \lr{\nabla f(x^k) -\nabla f(x^{k-1}), x^k - x^{k+1}} \leq
  2\la_k \n{\nabla f(x^k) -\nabla f(x^{k-1})} \n{x^k - x^{k+1}} \\ &=
  2\left(\la_k - \frac{\sqrt{\th_k}}{L}\right) \n{\nabla f(x^k) -\nabla f(x^{k-1})}\n{x^k -
      x^{k+1}} + \frac{2\sqrt{\th_k}}{L}\n{\nabla f(x^k) -\nabla f(x^{k-1})}\n{x^k - x^{k+1}}\\
  & \overset{\eqref{so_much_trouble}}{\leq} \frac{1}{L_k}
 \n{\nabla f(x^k) -\nabla f(x^{k-1})} \n{x^k -
    x^{k+1}} + \frac{\la_k\th_k}{L}\n{\nabla f(x^k)-\nabla f(x^{k-1})}^2 +
\frac{1}{\la_k L}\n{x^{k+1}-x^k}^2 \\
& =
\n{x^k -x^{k-1}} \n{x^k - x^{k+1}} + \frac{\la_k\th_k}{L}\n{\nabla f(x^k)-\nabla f(x^{k-1})}^2 +
\frac{1}{\la_k L}\n{x^{k+1}-x^k}^2 \\ & \leq
  \frac 1 2 \n{x^{k}-x^{k-1}}^2 + \frac{1}{2}\n{x^{k+1}-x^k}^2 + \frac{\la_k\th_k}{L}\n{\nabla f(x^k)-\nabla f(x^{k-1})}^2 +
\frac{1}{\la_k L}\n{x^{k+1}-x^k}^2.
\end{aligned}
\end{equation}
Combining everything together,  we obtain the statement of the  theorem.
\end{proof}

\section{Stochastic analysis}\label{ap:stoch}
\subsection{Different samples}
Consider the following version of SGD, in which we have two samples at each iteration, $\xi^k$ and $\zeta^k$ to compute
\begin{align*}
        \la_k &= \min\left\{\sqrt{1+\th_k} \la_{k-1}, \frac{\alpha\|x^k - x^{k-1}\|}{\|\nabla f_{\zeta^k}(x^k) - \nabla f_{\zeta^k}(x^{k-1})\|} \right\}, \\
        x^{k+1} &= x^k - \la_k \nabla f_{\xi^k}(x^k).
\end{align*}
As before, we assume that $\theta_0=+\infty$, so $\la_1 = \frac{\alpha\|x^1 - x^{0}\|}{\|\nabla f_{\zeta^1}(x^1) - \nabla f_{\zeta^1}(x^{0})\|}$.
\begin{lemma}
        Let $f_\xi$ be $L$-smooth $\mu$-strongly convex almost surely. It holds for $\lambda_k$ produced by the rule above
        \begin{align}
                \frac{\alpha}{L}
                \le \la_k
                \le \frac{\alpha}{\mu} \quad \text{a.s.} \label{eq:stochastic_la}
        \end{align}
\end{lemma}
\begin{proof}
        Let us start with the upper bound. Strong convexity of
        $f_{\zeta^k}$ implies that $\|x-y\|\le\frac{1}{\mu}\|\nabla
        f_{\zeta^k}(x) - \nabla f_{\zeta^k}(y)\|$ for any $x,
        y$. Therefore, $\la_k\le \min\left\{\sqrt{1+\th_k}\la_{k-1},
            \alpha/\mu \right\}\le \alpha/\mu$ a.s.

        On the other hand, $L$-smoothness gives
        $\la_k\ge \min\left\{\sqrt{1+\th_k} \la_{k-1},
            \alpha/L\right\}\ge \min\left\{ \la_{k-1},
            \alpha/L\right\}$ a.s. Iterating this inequality, we
        obtain the stated lower bound.
    \end{proof}
\begin{proposition}%
        Denote $\sigma^2\eqdef \E{\|\nabla f_\xi (x^*)\|^2}$ and assume $f$ to be almost surely $L$-smooth and convex. Then it holds for any $x$
        \begin{align}
                \E{\|\nabla f_\xi (x) \|^2}
                \le 4L (f(x) - f_*) + 2\sigma^2. \label{eq:sgd_variance}
        \end{align}
\end{proposition}
Another fact that we will use is a strong convexity bound, which states for any $x,y$
\begin{align}
        \<\nabla f(x), x - y>\ge \frac{\mu}{2}\|x-y\|^2 + f(x) -f(y). \label{eq:grad_dist_bound}
\end{align}
\begin{theorem}
        Let $f_\xi$ be $L$-smooth and $\mu$-strongly convex almost surely. If we choose some $\alpha \le \frac{\mu}{2L}$, then
        \begin{align*}
                \E{\|x^k - x^*\|^2}
                \le \exp\left(-k\frac{\mu\alpha}{L}\right)C_0 + \alpha\frac{\sigma^2}{\mu^2},
        \end{align*}
        where $C_0\eqdef 2(1+2\la_0^2L^2)\|x^0-x^*\|^2 + 4\la_0^2\sigma^2$ and $\sigma^2\eqdef \E{\|\nabla f_\xi(x^*)\|^2}$.
    \end{theorem}
\begin{proof}
        Under our assumptions on $\alpha\le \frac{\mu}{2L}$, we have $\la_k\le \frac{\alpha}{\mu}\le  \frac{1}{2L}$. Since $\la_k$ is independent of $\xi^k$, we have $\E{\la_k \nabla f_{\xi^k}(x^k)} = \E{\la_k}\E{\nabla f(x^k)}$ and
        \begin{align*}
                \E{\|x^{k+1} - x^*\|^2}
                &= \E{\|x^k - x^*\|^2} - 2 \E{\la_k\<\nabla f(x^k), x^k - x^*>} + \E{\la_k^2}\E{\|\nabla f_{\xi^k}(x^k)\|^2} \\
                &\overset{\eqref{eq:grad_dist_bound}}{\le} \E{(1-\la_k\mu)\|x^k - x^*\|^2} - 2\E{\la_k (f(x^k) - f_*)} + \E{\la_k^2}\E{\|\nabla f_{\xi^k}(x^k)\|^2} \\
                &\overset{\eqref{eq:sgd_variance}}{\le} \E{(1-\la_k\mu)\|x^k - x^*\|^2} - 2 \mathbb{E}\Bigl[\la_k\underbrace{(1 - 2\la_k L)}_{\ge 0}(f(x^k) - f_*)\Bigr] + \E{\la_k^2}\sigma^2 \\
                &\overset{\eqref{eq:stochastic_la}}{\le} \E{1-\la_k\mu}\E{\|x^k - x^*\|^2} + \alpha\frac{\E{\la_k}\sigma^2}{\mu}.
        \end{align*}
        Therefore, if we subtract $\alpha\frac{\sigma^2}{\mu^2}$ from both sides, we obtain
        \begin{align*}
                \E{\|x^{k+1} - x^*\|^2 - \alpha\frac{\sigma^2}{\mu^2}}
                \le \E{1 - \la_k\mu}\E{\|x^k - x^*\|^2 - \alpha\frac{\sigma^2}{\mu^2}}.
        \end{align*}
        If $\E{\|x^k-x^*\|^2}\le \alpha\frac{\sigma^2}{\mu^2}$ for some $k$, it follows that $\E{\|x^{t}-x^*\|^2}\le \alpha\frac{\sigma^2}{\mu^2}$ for any $t\ge k$. Otherwise, we can reuse the produced bound to obtain
        \begin{align*}
                \E{\|x^{k+1} - x^*\|^2}
                \le \prod_{t=1}^k\E{1 - \la_t\mu}\|x^1 - x^*\|^2 + \alpha\frac{\sigma^2}{\mu^2}.
        \end{align*}
        By inequality $1-x\le e^{-x}$, we have $\prod_{t=1}^k\E{1 - \la_t\mu}\le \exp\left(-\mu\sum_{t=0}^k\la_t \right)$.
        In addition, recall that in accordance with~\eqref{eq:stochastic_la}  we have $\la_k \ge \frac{\alpha}{L} $. Thus,
        \begin{align*}
                \E{\|x^{k+1} - x^*\|^2}
                \le \exp\left(-k\alpha\frac{\mu}{L}\right)\E{\|x^1 - x^*\|^2} + \alpha\frac{\sigma^2}{\mu^2}.
        \end{align*}
        It remains to mention that
        \begin{align*}
                \E{\|x^1-x^*\|^2} \le 2\|x^0-x^*\| + 2\la_0^2\E{\|\nabla f_{\xi^0}(x^0)\|^2} \overset{\eqref{eq:sgd_variance}}{\le} 2\|x^0-x^*\| + 2\la_0^2\left(2L^2\|x^0-x^*\|^2 + 2\sigma^2\right).
        \end{align*}
\end{proof}
This gives the following corollary.
\begin{corollary}
        Choose $\alpha = \gamma\frac{\mu}{2 L}$ with $\gamma\le 1$. Then, to achieve $\E{\|x^k - x^*\|^2}= \cO(\varepsilon + \gamma\sigma^2)$ we need only $k=\cO\left(\frac{L^2}{\gamma\mu^2}\log \frac{1+\la_0^2}{\varepsilon} \right)$ iterations. If we choose $\gamma$ proportionally to $\varepsilon$, it implies $\cO\left(\frac{1}{\varepsilon}\log \frac{1}{\varepsilon}\right)$ complexity.
    \end{corollary}
\subsection{Same sample: overparameterized models}
Assume additionally that the model is overparameterized, i.e.,\ $\nabla
f_\xi(x^*)=0$ with probability one.
In that case, we can prove that one can use the same stochastic sample to compute the stepsize and to move the iterate. The update becomes
\begin{align*}
        \la_k &= \min\left\{\sqrt{1+\th_k} \la_{k-1}, \frac{\alpha\|x^k - x^{k-1}\|}{\|\nabla f_{\xi^k}(x^k) - \nabla f_{\xi^k}(x^{k-1})\|} \right\}, \\
        x^{k+1} &= x^k - \la_k \nabla f_{\xi^k}(x^k).
\end{align*}
\begin{theorem}
        Let $f_\xi$ be $L$-smooth, $\mu$-strongly convex and satisfy $\nabla f_\xi(x^*)=0$  with probability one. If we choose $\alpha \le \frac{\mu}{L}$, then
        \begin{align*}
                \E{\|x^k - x^*\|^2}
                \le \exp\left(-k\alpha\frac{\mu}{L}\right)C_0,
        \end{align*}
        where $C_0\eqdef 2(1+\la_0^2L^2)\|x^0-x^*\|^2$.
\end{theorem}
\begin{proof}
        Now $\la_k$ depends on $\xi^k$, so we do not have an unbiased update anymore. However, under the new assumption, $\nabla f_{\xi^k}(x^*)=0$, so we can write
        \begin{align*}
                \<\nabla f_{\xi^k}(x^k), x^k-x^*>
                \overset{\eqref{eq:grad_dist_bound}}{\ge} \frac{\mu}{2}\|x^k -x^*\|^2 + f_{\xi^k}(x^k) - f_{\xi^k}(x^*).
        \end{align*}
        In addition, $L$-smoothness and convexity of $f_{\xi^k}$ give
        \begin{align*}
                \|\nabla f_{\xi^k}(x^k)\|^2
                \le 2L (f_{\xi^k}(x^k) - f_{\xi^k}(x^*)).
        \end{align*}
        Since our choice of $\alpha$ implies $\la_k\le \frac{1}{L}$, we conclude that
        \begin{align*}
                \|x^{k+1}-x^*\|^2
                &= \|x^k - x^*\|^2 - 2\la_k\<\nabla f_{\xi^k}(x^k), x^k- x^*> + \la_k^2\|\nabla f_{\xi^k}(x^k)\|^2 \\
                &\le (1 - \la_k\mu)\|x^k - x^*\|^2 - 2\la_k(1 - \la_k L)(f_{\xi^k}(x^k) - f_{\xi^k}(x^*)) \\
                &\le (1 - \la_k\mu)\|x^k - x^*\|^2.
        \end{align*}
        Furthermore, as $\|\nabla f_{\xi^0}(x^0)\|=\|\nabla f_{\xi^0}(x^0)- \nabla f_{\xi^0}(x^*)\|\le L\|x^0-x^*\|$, we also get a better bound on $\E{\|x^1-x^*\|^2}$, namely
        \begin{align*}
                \E{\|x^1-x^*\|^2}\le 2\|x^0-x^*\| + 2\la_0^2\E{\|\nabla f_{\xi^0}(x^0)\|^2} \le 2(1+\la_0^2L^2)\|x^0-x^*\|.
        \end{align*}
    \end{proof}

\section{Experiments details}\label{ap:exp_details}
Here we provide some omitted details of the experiments with neural
networks. We took the implementation of neural networks from a
publicly available repository\footnote{\href{https://github.com/kuangliu/pytorch-cifar/blob/master/models/resnet.py}{https://github.com/kuangliu/pytorch-cifar/blob/master/models/resnet.py}}. All methods were run with standard data augmentation and no weight decay. The confidence intervals for ResNet-18 are obtained from 5 different random seeds and for DenseNet-121 from 3 seeds.

In our ResNet-18 experiments, we used the default parameters for
Adam. SGD was used with a stepsize divided by 10 at epochs 120 and
160 when the loss plateaus. Log grid search with a factor of 2 was used to tune the initial stepsize of SGD and the best initial value was 0.2. Tuning was done by running SGD 3 times and comparing the average of test accuracies over the runs at epoch 200. For the momentum version (SGDm) we used the standard values of momentum and initial stepsize for training residual networks, 0.9 and 0.1 correspondingly. We used the same parameters for DenseNet-121 without extra tuning.

For our method we used the variant of SGD $x^{k+1}=x^k - \la_k \nabla f_{\xi^k}(x^k)$ with $\la_k$ computed using $\xi^k$ as well (biased option). We did not test stepsizes that use values other than $\frac{1}{L_k}$ and $\frac{1}{2L_k}$, so it is possible that other options will perform better. Moreover, the coefficient before $\th_{k-1}$ might be suboptimal too.
\end{document}